\documentclass[12pt]{amsart}
\usepackage{amscd}
\usepackage{amsmath}
\usepackage{comment}
\usepackage{color}
\usepackage[table]{xcolor}
\usepackage[all]{xy}

\usepackage{tikz-cd}
\usepackage{ytableau}

\usepackage{amssymb}

\numberwithin{equation}{section}

\newtheorem{thm}[equation]{Theorem}
\newtheorem{prop}[equation]{Proposition}

\newtheorem{lemma}[equation]{Lemma}
\newtheorem{cor}[equation]{Corollary}
\newtheorem{con}[equation]{Conjecture}

\theoremstyle{definition}
\newtheorem{rem}[equation]{Remark}
\newtheorem{example}[equation]{Example}
\newtheorem{dfn}[equation]{Definition}

\newcommand{\cP}{\mathcal{P}}

\newcommand{\Br}{\mathop{\mathrm{Br}}}

\newcommand{\ind}{\mathop{\mathrm{ind}}}

\newcommand{\CH}{\mathop{\mathrm{CH}}\nolimits}

\newcommand{\Spin}{\operatorname{\mathrm{Spin}}}

\newcommand{\GL}{\operatorname{\mathrm{GL}}}

\newcommand{\Ch}{\mathop{\mathrm{Ch}}\nolimits}

\newcommand{\res}{\mathop{\mathrm{res}}\nolimits}

\newcommand{\Char}{\mathop{\mathrm{char}}\nolimits}

\newcommand{\Z}{\mathbb{Z}}

\newcommand{\Q}{\mathbb{Q}}

\newcommand{\cT}{\mathcal T}

\newcommand{\Prod}{\operatornamewithlimits{\textstyle\prod}}

\renewcommand{\phi}{\varphi}

\newcommand{\ekbar}{\mathbf{e}}
\newcommand{\ck}{\mathbf{c}}
\newcommand{\pk}{\mathbf{p}}
\newcommand{\lk}{\mathbf{l}}

\newcommand{\vk}{\mathbf{v}}

\DeclareMathAlphabet{\cat}{OT1}{cmss}{m}{sl}

\newcommand{\abs}[1]{\left|#1\right|}

\title
{Counter-examples to a conjecture of Karpenko for spin groups}

\keywords
{
Algebraic groups;
Spin groups;
generic torsors;
projective homogeneous varieties;
Chow rings;
Grothendieck rings.
{\em Mathematical Subject Classification (2010):}
20G15; 14C25; 16E20}

\author
[S.~Baek]
{Sanghoon Baek}

\address
[S. Baek]
{Department of Mathematical Sciences\\
	KAIST\\
	291 Daehak-ro, Yuseong-gu\\
	Daejeon 305-701\\
	Republic of Korea}

\email
{sanghoonbaek@kaist.ac.kr}
\urladdr{mathsci.kaist.ac.kr/~sbaek}

\author
[R.~Devyatov]
{Rostislav Devyatov}

\address
[R.~Devyatov]
{Department of Mathematical Sciences\\
	KAIST\\
	291 Daehak-ro, Yuseong-gu\\
	Daejeon 305-701\\
	Republic of Korea}

\email
{deviatov@mccme.ru}
\urladdr{mccme.ru/~deviatov}

\date
{\today}

\thanks
{The work of both authors was supported by Samsung Science and Technology Foundation under Project Number SSTF-BA1901-02.}

\begin{document}

\begin{abstract}
Consider the canonical morphism from the Chow ring of a smooth variety $X$ to the associated graded ring of the topological filtration on the Grothendieck ring of $X$. In general, this morphism is not injective. However, Nikita Karpenko conjectured that these two rings are isomorphic for a generically twisted flag variety $X$ of a semisimple group $G$. The conjecture was first disproved by Nobuaki Yagita for $G=\Spin(2n+1)$ with $n=8, 9$. Later, another counter-example to the conjecture was given by Karpenko and the first author for $n=10$. In this note, we provide an infinite family of counter-examples to Karpenko's conjecture for any $2$-power integer $n$ greater than $4$. This generalizes Yagita's counter-example and its modification due to Karpenko for $n=8$.
\end{abstract}

\maketitle

\tableofcontents

%\addtocounter{section}{-1}

\section{Introduction}

For a smooth variety $X$ over a field $k$, let $\CH(X)$ and $K(X)$ denote the Chow and Grothendieck rings of $X$, respectively. Consider the associated graded ring $GK(X)$  of $K(X)$ with respect to the topological filtration, i.e.,
\begin{equation*}
GK(X)=\bigoplus_{i=0}^{\dim X}K(X)^{(i)}/K(X)^{(i+1)},
\end{equation*}
where $K(X)^{(i)}$ denotes the $i$-th term of the topological filtration of $K(X)$. 

The canonical morphism
\begin{equation}\label{eq:conjecture}
\phi: \CH(X)\to GK(X)	
\end{equation}
sending the class of a closed subvariety of $X$ in $\CH^{i}(X)$ to the class of its structure sheaf in $K(X)^{(i)}/K(X)^{(i+1)}$, is surjective but not injective in general. By
Riemann-Roch theorem, for all $i \ge 1$, the kernel of the $i$th homogeneous component 
\begin{equation*}
\phi^i : \CH^i(X) \to GK^i (X):= K(X)^{(i)}/K(X)^{(i+1)}
\end{equation*}
is annihilated by $(i-1)!$. Hence, the morphism $\phi$ becomes an isomorphism after tensoring with $\Q$. In particular, if $X$ is a flag variety, that is, the quotient $G/P$ of a split semisimple group $G$ by a parabolic subgroup $P$, then $\phi$ is an isomorphism as $\CH(X)$ is torsion-free. In \cite{Kar2017a}, Nikita Karpenko conjectured that the morphism $\phi$ is still injective for a generic flag variety $X$, namely:

\begin{con}\label{conjecture}
The morphism in $($\ref{eq:conjecture}$)$ is injective for a generic flag variety $X=E/P$ of a split semisimple group $G$, where $E$ denotes a generic $G$-torsor given by the generic fiber of a $G$-torsor $\GL(N)\to \GL(N)/G$ induced by an embedding of $G$ into a general linear group $\GL(N)$ for some $N\geq 1$ and $P$ denotes a parabolic subgroup of $G$.	
\end{con}

This conjecture has been verified in a number of cases, including simple groups $G$ of type $A$ and $C$ (see \cite[Theorem 1.2]{Kar2017}), special orthogonal groups $G$, the simply connected groups $G$ of type $G_{2}$, $F_{4}$, and $E_{6}$ (see \cite[Theorem 3.3]{Kar2017a}).

Now we consider the split spin group $G=\Spin(N)$ of a non-degenerate quadratic form of dimension $N$ over a field $k$. Let $P$ denote a maximal parabolic subgroup whose conjugacy class is obtained by the subset of the Dynkin diagram of $G$ corresponding to removing the last vertex. Then, a generic $G$-torsor $E$ gives rise to an $N$-dimensional generic quadratic form $q$ whose discriminant and Clifford invariant are trivial. The generic flag variety $X=E/P$ becomes a maximal orthogonal grassmannian of $q$. By \cite[Proposition 2.16]{BK} Conjecture \ref{conjecture} with $N=2n+1$ is equivalent to the same conjecture with $N=2n+2$. Thus, in this paper, we shall only consider the maximal orthogonal grassmannian $X$ with $N =2n+1$. 

Conjecture \ref{conjecture} holds for $1\leq n\leq 5$ (see \cite{Kar2018}). On the other hand, the conjecture was first disproved for $n=8, 9$ by Yagita \cite{Yagita}. Later, the counter-examples due to Yagita were extended to $n=8, 9, 10$ over the base field of arbitrary characteristic in \cite{Kar2020} and \cite{BK}. In the present paper, we 
generalize the proof for $n=8$ due to Karpenko and
construct an infinite family of counterexamples:

\begin{thm}\label{mainthm}
Let $n\geq 8$ be a power of $2$ and let $X$ be the maximal orthogonal grassmannian of a generic $n$-dimensional quadratic form with trivial discriminant and Clifford invariant. Then, the canonical epimorphism $\phi:\CH(X)\to GK(X)$ is not injective.
\end{thm}

For each 2-power $n \ge 8$, we construct an explicit element $x \in \CH(X)$ (see (\ref{eq:elementx}) below), which 
is not divisible by 2 in $\CH(X)$, but $\varphi(x)$ is divisible by 2 in $GK(X)$. In the following part of introduction, we sketch the proof that $x$ has these properties and provide some ideas behind the construction of such an element $x$. The detailed proof is given in later Sections \ref{sec:congruencerel} and \ref{sec:proofofthm}.

First, by \cite[Proposition 2.1]{Kar2018} the Chow ring $\CH(X)$ is generated by the Chern classes $c(1),\ldots, c(n)$ and an additional element $e\in \CH^{1}(X)$ (see Section \ref{subsection:chow}). Since the Chern classes satisfy the relations (\cite[Theorem 2.1]{Kar2018a}):
\begin{equation}
\label{relationinnobar}
c(i)^{2}=(-1)^{i+1}2c(2i)+2\sum_{k=1}^{i-1}(-1)^{k+1} c(i-k)c(i+k),
\end{equation}
we can rewrite any polynomial in $c(i)$ as a square-free polynomial. Hence, together with relation $c(1)=2e$, it suffices to consider an element $x$ of the form
\begin{equation}\label{xform}
x=e^s \prod_{j \in J} c(j) \in \CH^{l}(X), \text{ where $s\geq 0$ and $J$ is a subset of $[2,n]$}
\end{equation}
for some $l \geq 1$ or a linear combination of elements of this form. In this note, we focus on an element of the form (\ref{xform}).

In the proof of non $2$-divisibility of $x$, we make use of the degree map $\deg: \CH(X)\to \mathbb{Z}$ and the Steenrod operation $S$ on $\Ch(X):=\CH(X)/2\CH(X)$ following \cite{Kar2020} and \cite{BK}. In general, it is quite difficult to check the divisibility of an element in $\CH(X)$. 
However, 
the exact value of the index $\ind X$ 
(i.e., torsion index of $\Spin(2n+1)$) of $X$ is available: 
Let $r=\dim X=\frac{n(n+1)}{2}$. Then
\begin{equation*}
\ind X=2^{m}, 
	\text{ where } m=n-\lfloor \log_2(1+r)\rfloor \text{ or } m=n-\lfloor \log_2(1+r)\rfloor+1
\end{equation*}
(depending on $n$, see \cite{Totaro} for details). In particular, if $n$ is a power of 2, then the second formula for $m$ holds. So, 
it is often possible to determine the non-divisibility of an element in $\CH_{0}(X)=\CH^{r}(X)$ by $2$. 
Namely, since the image of the degree map is equal to $2^{m}\Z$, we get a well-defined homomorphism $2^{-m}\deg:\Ch(X)\to \Z/2\Z$. Moreover,
non-$2$-divisibility of an element $x$ of the form
(\ref{xform}) immediately follows from the non-triviality of the image $S(\bar{x})$ under the map $2^{-m}\deg$, where $\bar{x}$ denote the image of $x$ in $\Ch(X)$. Here, the use of Steenrod operations gives us more flexibility to find such an element $x$ that is non-divisible by $2$, while $\varphi (x)$ is divisible 
by $2$: using Steenrod operations, one can try to find such an element $x$ in an arbitrary graded component of $\CH(X)$.

The degree map is determined by the restriction map $\res:\CH(X)\to \CH(\bar{X})$, where $\bar{X}$ denotes the base change of $X$ to an algebraic closure of $k$. Hence, to show $2^{-m}\deg(S(\bar{x}))\neq 0$, it suffices to prove that 
\begin{equation}\label{keystep}
	\parbox{0.8\textwidth}{\noindent
		the image of an integral representative $x'$ of $S(\bar{x})$ under the restriction map is congruent to $2^{m}p$ modulo $2^{m+1}$,
	}
\end{equation}
where $p$ denote the class of a rational point. In fact, the congruence relation (\ref{keystep}) is the key step for the proof of non $2$-divisibility of $x$. For each $2$-power $n\geq 8$, this is proven in Proposition \ref{lem:indXplus1} by considering the element $x$ 
of the form
(\ref{xform}), where $l=r-3$, $s = n(\tfrac{n}4-1) + n-1$, 
and \begin{equation*}
J =\big([2,\tfrac{n}{4}+1] \cup [\tfrac{3n}{4}-1, n-1]\big) \setminus ( \{5\} \cup \{2^i \mid 2 \le i \le \log_{2}(n)-2\})
\end{equation*}
as in (\ref{eq:setJcases}).

From the formula (\ref{eq:steenrodformula}) ignoring the quadratic part and (\ref{eq:steenrode}), we can find an
integral representative $x'$, which is
a sum of elements of the same form as in 
(\ref{xform}), but with various numbers $s' \geq n(\tfrac{n}4-1) + n-1$ instead of $s$, and with various multi-subsets $J'$ of $[2, n]$ with $|J|=|J'|$ instead of $J$. Then, we check the divisibility of $\res(x') \in \CH(\bar{X})$ by $2$. The Chow ring $\CH(\bar{X})$ is generated by the special Schubert classes $e(1),\ldots, e(n)$ with the relations (\ref{relationinbar}) and the generators $c(i)$ and $e$ of $\CH(X)$ map to $2e(i)$ and $e(1)$ in $\CH(\bar{X})$, respectively, under the restriction map. Note that the relation (\ref{relationinbar}), as well as its powers, become simpler if considered modulo powers of $2$, which makes it easier to check the non-divisibility of an element of $\CH_0(\bar{X})$ by a power of $2$ compared to non-divisibility by other numbers.

Since $e(1)^{n} \equiv e(\tfrac{n}{2})^2 \mod 4$ (here we use the assumption that $n$ is a power of $2$), a direct calculation using a multinomial expansion of the $(\tfrac{n}4-1)$th power of the right-hand side of (\ref{relationinbar}) with $i=\tfrac{n}2$ yields that 
\begin{equation}\label{eq:eonepower}
e(1)^{n(\frac{n}4-1)}\equiv 
-(\tfrac{n}{4}-1) e(n) \cdot  2^{\frac{n}{4}-2}\big(\sum_{k=1}^{\frac{n}2-1}e(\tfrac{n}2-k)e(\tfrac{n}2+k) \big)^{\frac{n}{4}-2}
\mod 2^{\frac{n}2-\log_{2}(n)}, 
\end{equation}
and $e(1)^{n(\frac{n}4-1)}\equiv 0 \mod 2^{\frac{n}2-\log_{2}(n)-1}$ (see Lemma \ref{lemK:eonenj}). 

As for each $J'$ above we have $|J'|=|J|=\frac{n}2-\log_{2}(n)+3$, and $m=n-2\log_2(n)+2$, we see that each summand of $\res(x')$ becomes a multiple of $2^{m}$. 
Now, to conclude the proof, a careful computation is required to see from which multi-subsets $J'$ (and for which exponents $s'$) an extra multiple of $2$ arises. This is done in Corollary \ref{lem:eonepower} by multiplying (\ref{eq:eonepower}) by the Chern classes with indexes in $[\tfrac{3n}{4}-1, n-1]$, in Proposition \ref{lem:indXplus1}, in Remark \ref{lem:indXplus1n8}, and in Lemma \ref{lem:eonenjtorsioncor}.

Now, to show that $\phi(x)$ is divisible by $2$ in $GK(X)$, we use the Rees ring $\widetilde{K}(X)$ and its ideal $I(X)$ generated by $2$ and $t$ that are surjectively mapped onto $GK(X)$ and $2GK(X)$, respectively, by the map $\xi$ (see Section \ref{subsec:rees}). Since $x$ is of the form (\ref{xform}), by (\ref{imageunderphi}) and Lemma \ref{ecliffordtogkxnobar}, we have a standard preimage $w$ of $\varphi(x)$ in $\widetilde{K}(X)$ under $\xi$. By replacing the Chern classes $\ck(i)$ in $w$ with the element $2\ekbar(i)-t\ekbar(i+1)$ (see Lemma \ref{Lem:top}), we obtain 
another preimage $y\in \widetilde{K}(X)$ of $\varphi(x)$ under $\xi$ (i.e., $\xi(y)=\xi(w)=\varphi(x)$)
and show $y\in I(X)$.

In order to prove $y\in I(X)$, we view $y$ as contained in $\widetilde{K}(\bar{X})$ via the embedding $\widetilde{K}(X)\subset \widetilde{K}(\bar{X})$ and adopt an inductive argument as in \cite{Kar2020} and \cite{BK}. For any integers $l$ with $m>r-l\ge 0$ and $j\geq 1$, write 
\begin{equation*}
\widetilde{K}^{l}(\bar{X})\cap I(\bar{X})^{m+j}=2^{m+j}\widetilde{K}^{l}(\bar{X})+2^{m+j-1}t\widetilde{K}^{l+1}(\bar{X})+\cdots + 2^{m+j+l-r}t^{r-l}\widetilde{K}^{r}(\bar{X}).
\end{equation*}
Then, by the restriction-corestriction formula, 
$\ind X \cdot I(\bar X)\subset I(X)$ 
(see (\ref{eq:twoindextindex})). Hence, if $j < r - l$, we have modulo $I(X)$:
\begin{equation*}
\widetilde{K}^{l}(\bar{X})\cap I(\bar{X})^{m+j}\equiv 2^{m-1}t^{j+1}\widetilde{K}^{j+1}(\bar{X})+2^{m-2}t^{j+2}\widetilde{K}^{j+2}(\bar{X})+\cdots + 2^{m+j+l-r}t^{r-l}\widetilde{K}^{r}(\bar{X}).
\end{equation*}
For $j = r - l$, we simply get
$\widetilde{K}^{l}(\bar{X})\cap I(\bar{X})^{m+r-l}\subset I(X)$.

In the proof of Theorem \ref{mainthm}, we consider the case $l=r-3$ and $y\in \widetilde{K}^{l}(X)$ so that by (\ref{Kpl}) we get three congruence equations:
\begin{align*}
\widetilde{K}^{r-3}(\bar{X})\cap I(\bar{X})^{m+1}&\equiv \Z\cdot (2^{m-1}\lk)u^{r-3}+\Z\cdot (2^{m-2}\pk)u^{r-3}&\mod I(X),\\
\widetilde{K}^{r-3}(\bar{X})\cap I(\bar{X})^{m+2}&\equiv \Z\cdot (2^{m-1}\pk)u^{r-3} &\mod I(X),
\end{align*}
and $\widetilde{K}^{r-3}(\bar{X})\cap I(\bar{X})^{m+3}\subset I(X)$, where $\pk$ and $\lk$ denote the classes of a point and a line in $K(\bar{X})$. If the generators $(2^{m-1}\lk)u^{r-3}$ and $(2^{m-2}\pk)u^{r-3}$ are contained in $I(X)+I(\bar{X})^{m+2}$, then
\begin{multline}\label{eq:containmentK}
	\widetilde{K}^{r-3}(\bar{X}) \cap I(\bar{X})^{m+1} \subset \widetilde{K}^{r-3}(\bar{X})\cap (I(\bar{X})^{m+2} + I(X)) \subset \\
	\widetilde{K}^{r-3}(\bar{X})\cap (I(\bar{X})^{m+3} + I(X)) \subset I(X).
\end{multline} 
In addition, if $y$ is contained in $I(\bar{X})^{m+1}$, then by (\ref{eq:containmentK}) we conclude that $y\in I(X)$.

Alternatively, if $(2^{m-1}\lk)u^{r-3}, (2^{m-2}\pk)u^{r-3}\in I(X)$, then 
we could immediately conclude that
\begin{equation}\label{eq:ixbarembedding}
\widetilde{K}^{r-3}(\bar{X}) \cap I(\bar{X})^{m+1} \subset I(X).
\end{equation}

Consequently, the proof of $2$-divisibility of $\phi(x)$ is based on two main ingredients. The first one is to check that $y$ is contained in $I(\bar{X})^{m+1}$ (or a higher power of $I(\bar{X})$), which is proven in Proposition \ref{lemK:indXplus1} (a). This part is similar to the proof, as mentioned above, of the divisibility of each summand of $\res(x')\in \CH(\bar{X})$ by $2^{m}$. Indeed, some parts of the proof for $\res(x')$ even directly follow from the proof for $y$ because of a surjective morphism (\ref{morphismpsi}) from $\widetilde{K}(\bar{X})$ to $\CH(\bar{X})$.

The second ingredient is to show that some 
product of the class of a line or a point by a strict divisor of the torsion index is contained in $I(X)+ I(\bar{X})^{m+2}$
i.e., in our case $(2^{m-1}\lk)u^{r-3}, (2^{m-2}\pk)u^{r-3}\in I(X)+I(\bar{X})^{m+2}$. This is proven in Proposition \ref{lemK:indXplus1} (b) by slightly modifying $y$ into an element $z\in \widetilde{K}^{r-3}(X)$, which is congruent to $(2^{m-2}\lk)u^{r-3}$ modulo $I(\bar{X})^{m+1}$. As an additional consequence of Proposition \ref{lemK:indXplus1} (b), we indeed have $(2^{m-1}\lk)u^{r-3}, (2^{m-2}\pk)u^{r-3}\in I(X)$ (see Remark \ref{rem:proofKpart}). Therefore, we obtain (\ref{eq:containmentK}) and (\ref{eq:ixbarembedding}).

In this note, we focus on values of $n$ that are powers of $2$. This choice is advantageous for some arguments, such as the congruence relation $f(1)^{n}\equiv f(n) \mod I(X)$ given by (\ref{krelationinbarmodi}) and the property that the factorial $(\frac{n}{4})!$ is significantly more divisible by powers of $2$ than $(\frac{n}{4}-1)!$. However, the restriction to powers of $2$ is not always necessary for all arguments. We expect that the arguments requiring $n$ to be a power of $2$ can be extended to other values of $n$, and we plan to present generalizations in future publications, using examples from \cite{Totaro} of elements of $\operatorname{CH}(X)$ of top degree (i.e., of dimension $0$) that become divisible by $\ind X$ but not by $2 \ind X$ 
in $\operatorname{CH}(\bar{X})$.

So, throughout this note, $n$ is a power of $2$ and is bigger than $4$. We denote the integer interval $\{a, a+1, ..., b\}$ by $[a, b]$ for any $a\leq b$. If $b<a$, then $[a,b]$ denotes the empty set.

\section{Grothendieck and Chow rings of orthogonal grassmannians}

Throughout this paper, let $X$ denote the maximal orthogonal grassmannian (i.e., the variety of
$n$-dimensional totally isotropic subspaces) of a generic $(2n+1)$-dimensional quadratic form $q$ of trivial discriminant and Clifford invariant. The index of $X$, denoted by $\ind X$, is defined as the greatest common divisor of the degrees of closed points on $X$. Indeed, the index of $X$ is equal to the torsion index of $\Spin(2n+1)$, which is computed as follows (see \cite{Totaro}):
\begin{equation}\label{torsionindex}
	\ind X=2^{n-2v(n)+2},
\end{equation}
where $n$ is a power of $2$ and $v(n)$ denotes the exponent of $2$ in $n$.

\subsection{Grothendieck ring of orthogonal grassmannians}\label{subsec:Grothendieck}

Let $\bar{X}$ denote $X$ over an algebraic closure of $k$. In general,  by \cite{Panin} the ring $K(X)$ is identified with a subring of $K(\bar{X})$. As the Clifford invariant of $q$ is trivial, we have an isomorphism
\begin{equation}\label{Kidentification}
	K(X)=K(\bar{X}).
\end{equation}
The restriction map $K(X)^{(i)}\to K(\bar{X})^{(i)}$ is injective so that we view it as an inclusion:
\begin{equation}\label{Krestriction}
K(X)^{(i)}\subset K(\bar{X})^{(i)}
\end{equation}
for any $i\geq 1$. In particular, we have $K(X)^{(1)}=K(\bar{X})^{(1)}$. On the other hand, it follows by a restriction-corestriction argument that
\begin{equation}\label{Krescores}
 \ind X\cdot K(\bar{X})^{(i)}\subset K(X)^{(i)}
\end{equation}
for $i\geq 1$.

Write $\ck(i)\in K(X)^{(i)}$ for the $K$-theoretic Chern class of the dual of the (rank $n$) tautological vector bundle $\cT$ on $X$. Note that $\ck(i)=0$ for $i>n$. Let $\bar{Y}$ denote the quadric $Y$ of $q$ over an algebraic closure of $k$. We write $\ekbar(i)\in K(\bar{X})^{(i)}$ for the image of the class of a projective $(n-i)$-dimensional subspace $l_{n-i}$ on $\bar{Y}$ under the composition $(\pi_{1})_{*}\circ (\pi_{2})^{*}$ of the projective bundle $\pi_{1}\colon \cP\to \bar{X}$ given by the tautological vector bundle on $\bar{X}$ and the projection $\pi_{2}\colon \cP\to \bar{Y}$. 
We also set $\ekbar(i) = 0$ for $i > n$.
Then,
the following relations hold.
\begin{lemma}{\cite[Lemma 2.12]{BK}}\label{Lem:top}
	For any $i\geq 0$, the element
	\begin{equation*}
		2\ekbar(i)-\ekbar(i+1)-\ck(i)
	\end{equation*}
	is a sum of monomials in $\ck(1),\ldots, \ck(n)$ of degrees greater than or equal to $i+1$, where the degree of $\ck(j)$ for any $j\geq 0$ is defined to be $j$. In particular, $2\ekbar(i)-\ekbar(i+1)=\ck(i)\, \text{ in }\, 
	GK^i (X)
	$. 
\end{lemma}

Let us denote by $\pk$ and $\lk$ the classes of $\prod_{i=1}^{n}\ekbar(i)$ and $\prod_{i=2}^{n}\ekbar(i)$ in $K(\bar{X})^{(\dim\bar{X})}$ and $K(\bar{X})^{(\dim\bar{X}-1)}$, respectively. Then, we have
\begin{equation}\label{Kpl}
	K(\bar{X})^{(\dim\bar{X})}=\Z\cdot \pk \,\,\text{ and }\,\, K(\bar{X})^{(\dim\bar{X}-1)}=\Z\cdot \pk\oplus \Z\cdot \lk.
\end{equation}

\subsection{Rees ring associated to the topological filtration}\label{subsec:rees}

Consider the extended Rees ring $\widetilde{K}(X)$ of the Grothendieck ring $K(X)$ with respect to the topological filtration on $K(X)$, i.e.,
\begin{equation}\label{eq:wideReesK}
	\widetilde{K}(X)=\bigoplus_{i\in\Z}\widetilde{K}^i(X),\, \text{ where }\, \widetilde{K}^i(X)=K(X)^{(i)}t^{-i}
\end{equation}
for a variable $t$. Here we set $K(X)^{(i)} = K(X)$ for $i < 0$. Note also that $K(X)^{(i)} = 0$ for $i > \dim X$. We view $\widetilde{K}(X)$ as a subring of the Laurent polynomial ring $K(X)[t,t^{-1}]$. For notational simplicity, we write $u$ for $t^{-1}$. Observe that $t \in \widetilde{K}(X)$, while $u\not\in \widetilde{K}(X)$.

Let $I(X)$ denote the ideal of $\widetilde{K}(X)$ generated by $t$ and $2$. Then, we have
an isomorphism $\widetilde{K}(X)/t\widetilde{K}(X) \xrightarrow{\sim} GK(X)$. Denote the composition of the projection 
$\widetilde{K}(X) \to \widetilde{K}(X)/t\widetilde{K}(X)$ and this isomorphism by 
\begin{equation}
\label{eq:gkalternativeconstruction}
\xi \colon \widetilde{K}(X) \to GK(X).
\end{equation}
Note that then 
\begin{equation}
\label{eq:xiix2gkx}
\xi (I(X)) = 2GK(X).
\end{equation}

We define $\widetilde{K}(\bar{X})$ and $\bar \xi \colon \widetilde{K}(\bar X) \to GK(\bar X)$ in a similar way as in (\ref{eq:wideReesK}) and (\ref{eq:gkalternativeconstruction}), respectively. By (\ref{Krestriction}), we will treat $\widetilde{K}(X)$ as a 
subring of $\widetilde{K}(\bar{X})$.
Moreover, by (\ref{Krescores}) we have 
\begin{equation}\label{indexrestriction}
	\ind X\cdot\widetilde{K}(\bar{X})\subset \widetilde{K}(X).
\end{equation}
In particular,
\begin{equation}\label{eq:twoindextindex}
	2\ind X\cdot \widetilde K(\bar X),\,\,\,
t\ind X\cdot \widetilde K(\bar X) \subset I (X).
\end{equation}

Let $\bar{\phi}$ denote the morphism in (\ref{eq:conjecture}) for $\bar{X}$. As $\bar{X}$ is a flag variety, $\bar{\phi}$ is becomes an isomorphism. We shall denote by $\psi$ the composition
\begin{equation}\label{morphismpsi}
\psi: \widetilde{K}(\bar{X})
\stackrel{\bar \xi}{\longrightarrow}
GK(\bar{X})\stackrel{\bar \phi^{-1}}{\longrightarrow} \CH(\bar{X})
\end{equation}

We write 
\begin{equation*}
f(i)=\ekbar (i)u^{i}\in \widetilde K^{i}(\bar X) \text{ and } g(i)=2f(i)-tf(i+1)\in \widetilde K^{i}(\bar X)\cap I(\bar{X})
\end{equation*}
for all $i \in \mathbb{N}$. Then, by Lemma \ref{Lem:top}
\begin{equation}\label{eq:gikigicki}
g(i)\in \widetilde K^{i}(X)\,\, \text{ and }\,\, 
\xi (g(i)) = \xi(\ck(i)u^{i}).
\end{equation}
Moreover, by Corollary \ref{ktheorysquaresckcommas}, we have $f(n)^{2}=0$ and by Proposition \ref{ktheorysquaresckproducts}, 
the following relations hold modulo $I(\bar X)^{2}$:
\begin{equation}\label{krelationinbar}
	f(i)^{2}\equiv
	\begin{cases}
			(-1)^{i-1}f(2i)+tf(2i+1)+\sum\limits_{k=1}^{i-1}f(i+k)g(i-k)	& \text{ if } i \text{ is even},\\
		(-1)^{i-1}f(2i)+\sum\limits_{k=1}^{i-1}f(i+k)g(i-k) & \text{ if } i \text{ is odd}.
	\end{cases}
\end{equation}
Instead of using this formula in the whole generality, we shall use it either for $i \ge \tfrac{n}{2}$, or modulo $I(\bar X)$.
If $i \ge \tfrac{n}{2}$, then, since $f(2i+1)=0$, the relations (\ref{krelationinbar}) become
\begin{equation}\label{krelationinbarlargeimodsquare}
	f(i)^{2}\equiv
        (-1)^{i-1}f(2i)+\sum\limits_{k=1}^{i-1}f(i+k)g(i-k)	\mod I(\bar X)^{2},
\end{equation}
regardless of the parity of $i$. Modulo $I(\bar X)$,
we simply have
\begin{equation}\label{krelationinbarmodi}
	f(i)^{2}\equiv
        f(2i) \mod I(\bar{X})
\end{equation}
for any $i \in [1,n]$.

\subsection{Chow ring of orthogonal grassmannians}\label{subsection:chow}

Let $c(i)\in \CH^{i}(X)$ denote the Chern class of the dual of the tautological vector bundle $\cT$ and let $e(i)\in \CH^{i}(\bar{X})$ denote the image of the class $l_{n-i}\in \CH^{n-i}(\bar{X})$ of a projective $(n-i)$-dimensional subspace on $\bar{Y}$ under the composition $(\pi_{1})_{*}\circ (\pi_{2})^{*}$. Since the morphism $\phi$ in (\ref{eq:conjecture}) commutes with Chern classes, we have
\begin{equation}\label{imageunderphi}
\phi\big(c(i)\big)=\ck(i)+ K(X)^{(i+1)}.
\end{equation}
Moreover, the image of $e(i)$ under the isomorphism $\bar{\phi}$ is given by
\begin{equation}\label{imageunderphibar}
	\bar \phi\big(e(i)\big)=\ekbar(i)  + K(\bar X)^{(i+1)}.
\end{equation}

As an abelian group, $\CH(\bar{X})$ is freely generated by the set of all products of the form $\prod_{i\in I}e(i)$, where $I$ is an arbitrary subset of $[1, n]$. The Chow ring $\CH(\bar{X})$ is generated by $e(1),\ldots, e(n)$ subject to the relations
\begin{equation}
\label{relationinbar}
e(i)^{2}=(-1)^{i+1}e(2i)+2\sum_{k=1}^{i-1}(-1)^{k+1} e(i-k)e(  
i+k)
\end{equation}
for all $i\geq 1$, where we set $e(i)=0$ for $i>n$. In particular, we shall denote by $p$ the class of a rational point, i.e., $p=\prod_{i=1}^{n}e(i)\in \CH(\bar{X})^{(\dim\bar{X})}$. By \cite[Proposition 86.13]{EKM}, we have
\begin{equation}\label{eq:rescitwoe}
\res\big(c(i)\big)=2e(i)
\end{equation}
for all $1\leq i\leq n$, where $\res: \CH(X)\to \CH(\bar{X})$ denotes the restriction map.

By \cite[\S2]{MT}, there is an exact sequence of abelian groups:
\begin{equation*}
0\to \CH^{1}(X)\stackrel{\res}{\longrightarrow} \CH^{1}(\bar{X})\to \Br(k),
\end{equation*}
where the second map sends the generator $e(1)$ to the Brauer class of the even Clifford algebra of $q$. Since the Clifford invariant of $q$ is trivial, i.e., the Brauer class of the even Clifford algebra of $q$ is trivial, the restriction map is an isomorphism so that 
\begin{equation*}
\res(e)=e(1)
\end{equation*}
for some $e\in \CH^{1}(X)$. As $\res(c(1))=2e(1)$, we have $c(1)=2e$.

Since $K(X)^{(1)}=K(\bar X)^{(1)}$, the element $\ekbar(1) \in K(X)^{(1)}$ defines a class 
$\ekbar(1) + K(X)^{(2)}$ in $GK^{1}(X)$. In particular, we have
\begin{lemma}
\label{ecliffordtogkxnobar}
$\phi(e)=\ekbar(1)+ K(X)^{(2)}$.
\end{lemma}
\begin{proof}
Since $\varphi$ and $\bar \varphi$ commute with the field extension, we have the following commutative diagram:
\begin{equation*}
\xymatrix{
\CH^1(X) \ar[r]^{\varphi^1} \ar[d]^{\res} & GK^1(X) \ar[d]^{\res} \\
\CH^1(\bar X) \ar[r]^{\bar \varphi^1} & GK^1(\bar X).
}
\end{equation*}
Since all maps except the right vertical map are isomorphisms, the right vertical map $\res:GK^{1}(X)\to GK^{1}(\bar{X})$ is an isomorphism as well.

As $\res(\ekbar (1) + K(X)^{(2)}) = \ekbar (1) + K(\bar X)^{(2)}$ and $\bar \phi (\res(e)) = \ekbar (1) + K(\bar X)^{(2)}$ by (\ref{imageunderphibar}), both $\ekbar (1) + K(X)^{(2)}$ and $\varphi (e)$ have the same image under $\res:GK^{1}(X)\to GK^{1}(\bar{X})$, whence the proof follows.\end{proof}

Let $\Ch(X)$ denote the modulo $2$ Chow group, i.e., $\Ch(X):=\CH(X)/2\CH(X)$. For any $x\in \CH(X)$, we write $\bar{x}$ for the image of $x$ in $\Ch(X)$. Consider the total cohomological Steenrod operation $S:\Ch(X)\to \Ch(X)$ as in \cite{EKM} ($\Char{k}\neq 2$) and in \cite{Primozic} ($\Char{k}=2$). The operation commutes with pull-back morphisms, so it can be viewed as a contravariant functor from the category of smooth varieties to the category of abelian groups. Moreover, the Steenrod operation satisfies Cartan formula (\cite[Corollary 61.15]
{EKM}), i.e., it is a ring homomorphism.

For any $j\geq 0$, we denote by $S^{j}:\Ch^{*}(X)\to \Ch^{*+j}(X)$ the $j$th component of $S$. In particular, $S^{0}$ is the identity map. A formula for the values of $S^{j}$ on the Chern classes is given in \cite[Th\'eor\`eme 7.1]{Borel}:
	\begin{equation}\label{eq:steenrodformula}
S^j\big(\bar{c}(i)\big)=\binom{i-1}{j} \bar{c}(i+j)+Q(i,j)
\end{equation}
for any $i\geq0$ and $j\geq 1$, where $Q(i,j)$ denotes a linear combination of $\bar{c}(1)\bar{c}(i+j-1),\dots,\bar{c}(i)\bar{c}(j)$. We also have
\begin{equation}
\label{eq:steenrode}
	S(\bar{e})=\bar{e}+\bar{e}^{2}.
\end{equation}

\section{Congruence relations for split orthogonal grassmannians}\label{sec:congruencerel}

In this section, we shall compute some basic congruence relations in both the extended Rees ring $\widetilde K(\bar X)$ and the Chow ring $\CH(\bar{X})$.

Let us recall some basic notions concerning multisets and introduce some specific notations. A \emph{multiset} is an unordered collection of elements with duplicates allowed. The \emph{cardinality} of a multiset $J$ is the sum of the multiplicities of all its elements and is denoted by $|J|$. The \emph{sum} of two multisets $J$ and $L$, denoted by $J+L$, is the multiset such that the multiplicity of an element is equal to the sum of the multiplicities of the element in $J$ and $L$. 
We say that a multiset $J$ is a \emph{multi-subset} of a set $S$ and write 
$J\subset S$, if every element of $J$ is an element of $S$ (note that we allow multiplicities greater than 1 in $J$ here). 
For any multi-subset $J$ of $[1,n]$,  we write 
\begin{equation*}
	\ekbar(J)=\Prod_{j \in J} \ekbar(j) \in K(\bar X) \text{ and } e(J)=\Prod_{j \in J} e(j) \in \CH(\bar X).
\end{equation*}
Similarly, we write
\begin{equation*}
	f(J)=\Prod_{j \in J} f(j)
	\in \widetilde K^{|J|}(\bar X) \text{ and } g(J)=\Prod_{j \in J} g(j)
	\in \widetilde K^{|J|}(\bar X).
\end{equation*}

For a nonzero element $a\in \widetilde K(\bar X)$, we write $\vk(a)$ for the highest power of $I(\bar{X})$ containing $a$. Similarly, for a nonzero element $b$ in $\mathbb{Z}$ or $\CH(\bar{X})$ we write $v(b)$ for the highest power of $2$ dividing $b$.

We shall write 
\begin{equation*}
I_{0}:=[\tfrac{n}{2}+1, n-1]=[\tfrac{4n}{8}+1, \tfrac{5n}{8}]\cup [\tfrac{5n}{8}+1, \tfrac{6n}{8}-2]\cup [\tfrac{6n}{8}-1, n-1]=:I_{1}\cup I_{2}\cup I_{3}
\end{equation*}
and $\bar{I}_{3}=I_{3}\cup \{n\}$. 
We set $I_1= \emptyset$ for $n=8$.

In the following, we find some congruence relations modulo certain powers of $I(\bar X)$ and $2$, respectively, for some elements of $\widetilde{K}(\bar X)$ and $\CH(\bar X)$ that can be written as products of $f(i)$'s and $g(i)$'s, and of $e(i)$'s, respectively. We start with powers of a single factor $f(i)$ or $e(i)$.

\begin{lemma}\label{lemK:eij}
	Let $i\in I_{0}$ and $j \in \mathbb{N}$
	be integers. Then, 
	\begin{equation}\label{eq:lemK:eij}
		f(i)^{j}\equiv \sum a(J) f(J) \mod I(\bar X)^{v(j!)+1}\,\, \text{ and }\,\, e(i)^{j}\equiv 2^{v(j!)}\sum e(J) \mod 2^{v(j!)+1},
	\end{equation}
	where $a(J)\in I(\bar X)^{v(j!)}$ and
	the sums range over some multi-subsets $J\subset [1,n]$ such that $|J|=j$. 
	In particular, 
	\begin{equation*}
		\vk\big(f(i)^{j}\big),\, v\big(e(i)^{j}\big)\geq v(j!).
	\end{equation*}
	Moreover, if $j \geq 2$, then the multisets $J$ above satisfy $J\,\cap\, [i+1,n]\neq \emptyset$.
	
	Furthermore, if $i\in I_{2}$ and $j \geq 2$, then the same relations $($\ref{eq:lemK:eij}$)$ hold, where the sum ranges over some multi-subsets $J$ with $|J|=j$, $J\,\cap\, [i+1,n]\neq \emptyset$, and such that either $J\subset I_{1}\cup I_{2}$ or $J\cap \bar{I}_{3}\neq \emptyset$.
\end{lemma}
\begin{proof}
Instead of proving the lemma as it is stated, claiming simply that $a(J)\in I(\bar X)^{v(j!)}$, let us prove a stronger statement: $a(J)$ is a sum of 
terms of the form
\begin{equation}\label{eq:aform}
2^q t^{v(j!)-q} \text{ for some } q \in [0, v(j!)].
\end{equation}

If $j = 1$, then the statement is trivial. For the case of arbitrary $j\geq 2$ 
	we show the first equation in (\ref{eq:lemK:eij}). Let $i\in I_{0}$. By the binary expansion of $j$, it suffices to prove the statement for any integer $j$ that is a power of $2$. We prove by induction on $j\geq 2$. Assume that $j=2$. Then, as $\ekbar(2i)=\ekbar(2i+1)=0$ for any $i\in I_{0}$, we have $f(2i)=f(2i+1)=0$, thus the statement follows by (\ref{krelationinbarlargeimodsquare}). 
	Assume that the statement holds for $j$. Then, modulo $I(\bar X)^{2v(j!)+2}$ we have
	\begin{equation}\label{eqK:sumjj}
		f(i)^{2j}\equiv\big(\sum_{J} a(J) f(J)\big)^{2}=\sum_{J} a(J)^{2} f(J)^{2}+\sum_{J\neq J'} 2a(J)^{2} f(J+J').
	\end{equation}
	Let $k\in J\cap [i+1, n]$ and $J^{c}=J-\{k\}$. Then, the case $j=2$ implies that 
	\begin{equation*}
		f(J)^{2}=f(k)^{2}f(J^{c}+J^{c})=\sum_{L} b(L) f(L)f(J^{c}+J^{c})=\sum_{L} b(L) f(L+J^{c}+J^{c}),
	\end{equation*}
	where $L$ denotes a multi-subset such that $|L|=2$ and $L\cap [k+1, n]\neq \emptyset$, and $b(L)=2$ or $t$. Since $2v(j!)+1=v\big((2j)!\big)$, each summand in (\ref{eqK:sumjj}) satisfies the statement. The same proof works in the case $i\in I_{2}$.
	
Furthermore, since $a(J)$ is a sum of terms of the form (\ref{eq:aform}), we have 
$\psi(a(J)) \equiv 2^{v(j!)} \mod 2^{v(j!)+1}$ 
or
$\psi(a(J)) \equiv 0 \mod 2^{v(j!)+1}$, where $\psi$ denotes the morphism in (\ref{morphismpsi}), so the second equation in (\ref{eq:lemK:eij}) follows.
\end{proof}

As a corollary of this lemma, we can observe the behavior of powers of $f(i)g(n-i)$ after the multiplication by $f(\bar{I}_3)$.

\begin{cor}\label{lemK:eijIthree}
	Let $j\geq 2$ be an integer. Then, modulo $I(\bar X)^{v(j!)+j+1}$ we have 
	\begin{equation*}
	f(i)^{j}\cdot g(n-i)^{j}\cdot f(\bar{I}_{3})\equiv
	\begin{cases}
	0 & \text{ if i}\in I_{1},\\
	\sum a(J) f(J) f(\bar{I}_{3})      & \text{ if i}\in I_{2}     
	\end{cases}
	\end{equation*}
	for some $a(J)\in I(\bar X)^{v(j!)+j}$, where the sum ranges over some multi-subsets $J\subset I_{1}\cup I_{2}$ with $|J|=j+1$.
\end{cor}
\begin{proof}
	For $j\geq 2$, the binomial expansion of $g(n-i)^{j}$ tells us that each summand (modulo $I(\bar X)^{j+1}$) 
	is divisible by 
	$f(n-i)^{2}$ or $f(n-i+1)^{2}$ and moreover, it can be written in the form
	$bf(n-i)^{2}$ or $bf(n-i+1)^{2}$ for some $b\in I(\bar X)^{j}$. 
	It follows by (\ref{krelationinbarmodi})
	that
	\begin{equation*}
	f(n-i)^{2}\equiv f(2n-2i),\,\,  f(n-i+1)^{2}\equiv f(2n-2i+2)
	\mod I(\bar X)
	\end{equation*}   
	for any $i\in I_{1}\cup I_{2}$.

	Assume that $i\in I_{1}$. Then, $2n-2i,\, 2n-2i+2\in \bar{I}_{3}$, thus we get
	\begin{equation*}
	bf(n-i)^{2}\cdot f(\bar{I}_{3})\equiv bf(n-i+1)^{2}\cdot f(\bar{I}_{3})\equiv 0 \mod I(\bar X)^{j+1}.
	\end{equation*}
	Hence, by Lemma \ref{lemK:eij} each summand of $f(i)^{j}g(n-i)^{j}f(\bar{I}_{3})$ is contained in $I(\bar X)^{v(j!)+j+1}$.

	Now we assume that $i\in I_{2}$. Then $2n-2i \in I_1 \cup I_2$ and $2n-2i +2 \in I_1 \cup I_2 \cup \{\tfrac{6n}8\}$. If $2n-2i+2 = \tfrac{6n}8$, then again, 
		by Lemma \ref{lemK:eij}, the summands of $f(i)^{j}g(n-i)^{j}f(\bar{I}_{3})$ 
		divisible by
		$f(n-i+1)^{2}$
		are contained in $I(\bar X)^{v(j!)+j+1}$.
	
	Consider a summand of $g(n-i)^j$ of the form $bf(n-i)^{2}$ or $bf(n-i+1)^{2}$ 
		with $2n-2i \in I_1 \cup I_2$ or $2n-2i +2 \in I_1 \cup I_2$, respectively. 
		Let us still rewrite it (modulo $I(\bar X)^{j+1}$) as $bf(2n-2i)$ or $bf(2n-2i+2)$. By Lemma \ref{lemK:eij}, we get
	\begin{equation*}
	f(i)^{j}\equiv \sum a(J') f(J') \mod I(\bar X)^{v(j!)+1}
	\end{equation*}
	for some $a(J')\in I(\bar X)^{v(j!)}$, where the sum ranges over some multi-subsets $J'$ with $|J'|=j$ such that either $J'\subset I_{1}\cup I_{2}$ or $J'\cap \bar{I}_{3}\neq \emptyset$. Set $J=J'+ \{2n-2i\}$ or $J'+\{2n-2i+2\}$ and $a(J)=b\cdot a(J')$. Then, as $f(J')\cdot f(\bar{I}_{3})\equiv 0 \mod I(\bar X)$ for any $J$ with $J\cap I_{3}\neq \emptyset$, the statement follows.\end{proof}

Let us use these results to express powers of $f(1)^{n}$ in terms of powers of $f(i)$ and $g(i)$, and powers of $e(1)^{n}$ in terms of powers of $e(i)$ with different values of $i$.

\begin{lemma}\label{lemK:eonenj}
	For any $j\geq 2$, we have $f(1)^{nj} \in I(\bar X)^{j + v(j!) - 1}$ and
	\begin{equation}\label{lemK:eonenj_eq1}
	f(1)^{nj}\equiv -j\cdot \big(\sum_{i\in I_{0}} f(i)g(n-i)\big)^{j-1}\cdot f(n) \mod I(\bar X)^{j+v(j!)}.
	\end{equation}
	Also, we have $e(1)^{nj}\equiv 0 \mod 2^{j+v(j!)-1}$ and
	\begin{equation*}
		e(1)^{nj}\equiv -j\cdot 2^{j-1} \big(\sum_{i\in I_{0}} e(i)e(n-i)\big)^{j-1}\cdot e(n) \mod 2^{j+v(j!)}.
	\end{equation*}
	In particular,
	\begin{equation*}
		\vk(f(1)^{\frac{n^{2}}{4}}),\, v(e(1)^{\frac{n^{2}}{4}})\geq \tfrac{n}{2}-2\,\, \text{ and }\,\, \vk(f(1)^{\frac{n^{2}}{4}-n}),\, v(e(1)^{\frac{n^{2}}{4}-n}) \geq \tfrac{n}{2}-1-v(n).
	\end{equation*}
\end{lemma}
\begin{proof}
	Let $h=\sum_{i\in I_{0}} f(i)g(n-i)$. For any $k\geq 1$, consider the multinomial expansion 
	$h^{k}=\sum C(k_{\frac{n}{2}+1}, \ldots, k_{n-1})$, where the sum runs over $(k_{\frac{n}{2}+1}, \ldots, k_{n-1})\in (\mathbb{N}\cup \{0\})^{\frac{n}{2}-1}$ with $\sum_{i\in I_{0}}k_{i}=k$ and
	\begin{equation}\label{eqk:bjmulti}
		C(k_{\frac{n}{2}+1}, \ldots, k_{n-1})={k \choose k_{\frac{n}{2}+1},\ldots, k_{n-1}}\prod_{i\in I_{0}}f(i)^{k_{i}}g(n-i)^{k_{i}}.
	\end{equation}
	Then, by Lemma \ref{lemK:eij}
	\begin{equation}\label{eqk:vbj}
		v({k \choose k_{\frac{n}{2}+1},\ldots, k_{n-1}})+\sum_{i\in I_{0}} \vk(f(i)^{k_{i}})\geq v(k!), \text{ thus } \vk(h^{k})\geq v(k!)+k.
	\end{equation}

	By (\ref{krelationinbarmodi}), 
	$f(1)^{\frac{n}{2}}\equiv f(\frac{n}{2}) \mod I(\bar X)$
	and by (\ref{krelationinbarlargeimodsquare}),
	$f(\frac{n}{2})^{2}\equiv h-f(n) \mod I(\bar X)^{2}$, thus
	\begin{equation*}
		f(1)^{n}\equiv h-f(n) \mod I(\bar X)^{2}.
	\end{equation*}
	Write $f(1)^{n}=a+h-f(n)$ for some $a\in I(\bar X)^{2}$. As $f(n)^{2}=0$, we have 
	\begin{equation}\label{eqk:e1nj}
		f(1)^{nj}=\sum_{k=0}^{j}{j \choose j-k, k}a^{j-k}h^{k} - 
		f(n)\sum_{k=0}^{j-1}{j \choose j-k-1, 1, k}a^{j-k-1}h^{k}.
	\end{equation}
	Since $j-k\geq v((j-k)!)$ for $j-k \geq 0$,
	it follows from (\ref{eqk:vbj}) that each summand of the first sum in (\ref{eqk:e1nj}) is contained in $I(\bar X)^{j+v(j!)}$. Similarly, as $j - k - 2 \geq v((j- k - 1)!)$ for $k < j-1$, 
	each summand of the second sum in (\ref{eqk:e1nj}) is contained in $I(\bar X)^{j+v(j!)}$ except for the last term, which completes the proof of the equation (\ref{lemK:eonenj_eq1}).

	After we get (\ref{lemK:eonenj_eq1}), it follows
	from
	(\ref{eqk:vbj}) with $k = j-1$ that 
	\begin{equation*}
	f(1)^{nj} \in I(\bar X)^{v(j) + v((j-1)!) + j - 1} =  I(\bar X)^{j + v(j!) - 1}.
	\end{equation*}
	
	The statements for $\CH(\bar X)$ are obtained from the statements for $\widetilde K(\bar X)$ by applying $\psi$ in (\ref{morphismpsi}). The last statement immediately follows from 
	\begin{equation}\label{eq:vnoverfour}
		v\big((\tfrac{n}{4})!\big)=\tfrac{n}{4}-1,\,\,	
		v\big((\tfrac{n}{4}-1)!\big)=\tfrac{n}{4}-1-v(\tfrac{n}{4})
		=\tfrac{n}{4}+1-v(n).\qedhere
	\end{equation}
\end{proof}

Now let us obtain an expression for the product of certain high powers of $f(1)$ and $f(I_3)f(\frac{n}{2})$ in $\widetilde{K}(\bar X)$, and a similar result in $\CH(\bar X)$. Roughly speaking, what we are going to do is to express (modulo powers of $I(\bar{X})$ and of 2) the powers of sums in right-hand sides of the formulas in Lemma \ref{lemK:eonenj} as square-free products of $f(i)$'s and $g(i)$'s and of $e(i)$'s, respectively.

\begin{prop}\label{propK:eoneeIthree}
	For any $n\geq 8$, we have $f(1)^{\frac{n^2}{4}-n} \cdot f(I_3) \cdot f(\tfrac{n}{2}) \in I(\bar X)^{\frac{n}{2}-v(n)-1}$ and
	\begin{equation*}
		f(1)^{\frac{n^{2}}{4}-n}\cdot f(I_{3})\cdot f(\tfrac{n}{2})\equiv 
		-(\tfrac{n}{4} - 1)!\cdot f([\tfrac{n}{2}, n])\cdot g([\tfrac{n}{4}+2, \tfrac{n}{2}-1])  \mod I(\bar X)^{\frac{n}{2}-v(n)}
	\end{equation*}
Also, we have $e(1)^{\frac{n^2}{4}-n} \cdot e(I_3) \cdot e(\tfrac{n}{2})\equiv 0 \mod 2^{\frac{n}{2}-v(n)-1}$ and 
		\begin{equation*}
			e(1)^{\frac{n^{2}}{4}-n}\cdot e(I_{3})\cdot e(\tfrac{n}{2})\equiv 
			-(\tfrac{n}{4} - 1)!\cdot 2^{\frac{n}{4}-2} e([\tfrac{n}{4}+2, n]) \mod 2^{\frac{n}{2}-v(n)}.
		\end{equation*}
\end{prop}
\begin{proof}
	Let $k=\frac{n}{4}-2$. Consider a summand $C(k_{\frac{n}{2}+1}, \ldots, k_{n-1})$ of the multinomial expansion of $h^{k}$ as in (\ref{eqk:bjmulti}). We first show that 
	\begin{equation}\label{eqk:cknover}
		C(k_{\frac{n}{2}+1}, \ldots, k_{n-1})\cdot f(\bar{I}_{3})\equiv 0 \mod I(\bar X)^{v(k!)+k+1}
	\end{equation}
	for all $k_{\frac{n}{2}+1}, \ldots, k_{n-1}$ except for $k_{i}=\begin{cases} 1 & \text{ if } i\in I_{1}\cup I_{2},\\
		0 & \text{ if } i\in I_{3}.\end{cases}$

	If $k_{l}>0$ for some $l\in I_{3}$, then by Lemma \ref{lemK:eij} 
	\begin{equation*}
		f(l)^{k_{l}} f(\bar{I}_{3})=f(l)^{k_{l}+1}f(\bar{I}_{3}- \{l\})\equiv \sum_{J} a(J) f(J)f(\bar{I}_{3}- \{l\})
		\mod I(\bar X)^{v((k_{l}+1)!)+1},
	\end{equation*}
	where $|J| = k_{l} + 1$, $J \cap [l+1, n] \ne \emptyset$, and	
	$a(J)\in I(\bar X)^{v((k_{l}+1)!)}$. 
	Since $f(J)f(\bar{I}_{3}- \{l\})\equiv 0 \mod I(\bar X)$ by (\ref{krelationinbarmodi}), 
	we get $f(l)^{k_{l}}f(\bar{I}_{3})\equiv 0 \mod I(\bar X)^{v(k_{l}!)+1}$, thus again by Lemma \ref{lemK:eij}
	\begin{equation}\label{vkchoosevf}
	v({k \choose k_{\frac{n}{2}+1},\ldots, k_{n-1}})+\vk(f(l)^{k_{l}} f(\bar{I}_{3})) +\sum_{i\in I_{0} \setminus \{l\}} \vk(f(i)^{k_{i}})\geq v(k!)+1.
	\end{equation}
	Hence, 
	\begin{equation}\label{vkchoosevfg}
	v({k \choose k_{\frac{n}{2}+1},\ldots, k_{n-1}})+\vk(f(l)^{k_{l}} f(\bar{I}_{3}) g(l)^{k_{l}} ) +
	\sum_{i\in I_{0} \setminus \{l\}} \vk(f(i)^{k_{i}} g(i)^{k_{i}})\geq v(k!)+k+1,
	\end{equation}
	and the congruence (\ref{eqk:cknover}) holds. Therefore, we may assume that $k_{i}=0$ for all $i\in I_{3}$ and $\sum_{i\in I_{1}\cup I_{2}}k_{i}=k$.
	
	Similarly, if $k_{l}\geq 2$ for some $l\in I_{1}$, then by Corollary \ref{lemK:eijIthree} and Lemma \ref{lemK:eij}, we get (\ref{vkchoosevfg}), thus the congruence (\ref{eqk:cknover}) follows.

	Now, if $k_{l}\geq 2$ for some $l\in I_{2}$, then again by Lemma \ref{lemK:eij} and Corollary \ref{lemK:eijIthree}
	\begin{equation}\label{eqk:prodfsingleg}
		\big(\prod_{i\in I_{1}\cup I_{2}} f(i)^{k_{i}}\big)\cdot g(n-l)^{k_l}\cdot f(\bar{I}_{3})\equiv 
		f(\bar{I}_{3}) \prod_{i\in I_{1}\cup I_{2}} \sum_{J_{i}} a(J_i) f(J_{i}) \mod I(\bar X)^{s+1},
	\end{equation}
	where 
	$s=\sum_{i\in I_{1}\cup I_{2}} v(k_{i}!)+k_l$, $J_i \subset I_{1}\cup I_{2}$ or $J_i \cap \bar{I}_{3}\neq \emptyset$,
	\begin{equation*}
|J_{i}|=\begin{cases}
k_{i}&\text{ if } i\in (I_{1}\cup I_{2})\backslash\{l\},\\
k_{l}+1 & \text{ if } i=l,
\end{cases} \text{ and }a(J_{i})\in\begin{cases}
	 I(\bar X)^{v(k_i!)} &\text{ if } i\in (I_{1}\cup I_{2})\backslash\{l\},\\
	 I(\bar X)^{v(k_l!)+k_l} & \text{ if } i=l.
	\end{cases}
	\end{equation*}

	Since for each $k$-tuple $(J_i)_{i \in I_i \cup I_2}$
	\begin{equation*}
	\sum_{i\in I_{1}\cup I_{2}} |J_{i}|=1+\sum_{i\in I_{1}\cup I_{2}} k_{i}> |I_{1}\cup I_{2}|=k,
	\end{equation*}
	by (\ref{krelationinbarmodi})
	we obtain
	\begin{equation*}
		f(\bar{I}_{3}) \prod_{i\in I_{1}\cup I_{2}} f(J_{i}) \in I(\bar X).
	\end{equation*}
	Thus, the product of the sums on the right-hand side of (\ref{eqk:prodfsingleg}) belongs to $I(\bar X)^{s+1}$. As
		\begin{equation*}
			\prod_{i \in (I_{1}\cup I_{2}) \setminus \{l\}} g(n-i)^{k_i} \in I(\bar X)^{k - k_l},
		\end{equation*}
	the congruence (\ref{eqk:cknover}) follows. Therefore,
	\begin{equation*}
		h^{k}\cdot f(\bar{I}_{3}) \cdot f(\tfrac{n}{2})\equiv k!\cdot f([\tfrac{n}{2}, n])\cdot g([\tfrac{n}{4}+2, \tfrac{n}{2}-1])  \mod I(\bar X)^{v(k!)+k+1}
	\end{equation*}
    Hence, the second equation in the statement follows from 
    Lemma \ref{lemK:eonenj} and (\ref{eq:vnoverfour}) with the equality $v\big((\tfrac{n}{4}-2)!\big)=v\big((\tfrac{n}{4}-1)!\big)$. 
    Since 
    \begin{equation*}
    |[\tfrac{n}{4}+2, \tfrac{n}{2}-1]| = \tfrac{n}{4}-2 \text{ and }
    g([\tfrac{n}{4}+2, \tfrac{n}{2}-1]) \in 
    I(\bar X)^{\frac{n}{4}-2},
    \end{equation*}
    the first equation in the statement follows from the second equation.
	
The remaining equations in the statement follow from the first and second equations by applying $\psi$ in (\ref{morphismpsi}) together with 
(\ref{imageunderphibar}).
\end{proof}

As a corollary, we also obtain an expression for the product of $f(1)^{\frac{n^{2}}{4}-n}$ and $g(I_3)f(\tfrac{n}{2})$.

\begin{cor}\label{cor:fonepower}
For any $n\geq 8$, we obtain
$f(1)^{\frac{n^{2}}{4}-n}\cdot g(I_{3}) \in I(\bar X)^{\frac{3n}{4}-v(n)}$ and 
\begin{equation*}
f(1)^{\frac{n^{2}}{4}-n}\cdot g(I_{3})\cdot f(\tfrac{n}{2})\equiv 
2^{\frac{n}{2}-v(n)+2}\cdot f([\tfrac{n}{2}, n])\cdot g([\tfrac{n}{4}+2, \tfrac{n}{2}-1])  \mod I(\bar X)^{\frac{3n}{4}-v(n)+1}.
\end{equation*}	
\end{cor}
\begin{proof}
The first statement follows immediately from the last statement of Lemma \ref{lemK:eonenj}. For the second statement,
we show that
\begin{equation}\label{eq:fonensquare}
f(1)^{\frac{n^{2}}{4}-n} g(I_{3})\equiv f(1)^{\frac{n^{2}}{4}-n}\cdot 2^{\frac{n}{4}+1}\cdot f(I_{3}) \mod I(\bar X)^{\frac{3n}{4}-v(n)+1}.
\end{equation}
Then, the second statement immediately follows by Proposition \ref{propK:eoneeIthree} and (\ref{eq:vnoverfour}).

Write the left-hand side of the equation (\ref{eq:fonensquare}) as 
\begin{equation*}
f(1)^{\frac{n^{2}}{4}-n} g(I_{3})=f(1)^{\frac{n^{2}}{4}-n}\cdot  g(I_{3}\backslash \{n-1\})\cdot \big( 2f(n-1)-tf(n)\big).
\end{equation*}
Since $f(n)\equiv f(1)^{n} \mod I(\bar{X})$ by (\ref{krelationinbarmodi}), 
it follows by Lemma \ref{lemK:eonenj} that
\begin{equation*}
f(1)^{\frac{n^{2}}{4}-n}f(n)\equiv 0 \mod I(\bar{X})^{\frac{n}{2}-v(n)}.
\end{equation*}
As $|I_{3}|=\tfrac{n}{4}+1$, we have $t\cdot g(I_{3}\backslash \{n-1\})\in I(\bar{X})^{\frac{n}{4}+1}$, thus
\begin{equation*}
f(1)^{\frac{n^{2}}{4}-n}\cdot g(I_{3})\equiv 2f(1)^{\frac{n^{2}}{4}-n}\cdot g(I_{3}\backslash \{n-1\})f(n-1)  \mod I(\bar X)^{\frac{3n}{4}-v(n)+1}.
\end{equation*}

Let us expand the term $g(I_{3}\backslash \{n-1\})f(n-1)$. Then, each summand has a factor of the form $f(J)$ for some multi-subset $J\subset I_{3}$ with $|J|=|I_{3}|$. Since $f(j)^{2}\equiv 0 \mod I(\bar{X})$ for any $j\in I_{3}$, 
\begin{equation*}
	g(I_{3}\backslash \{n-1\})f(n-1)\equiv 2^{\frac{n}{4}}\cdot f(I_{3}) \mod I(\bar{X})^{\frac{n}{4}+1},
\end{equation*}
whence the equation (\ref{eq:fonensquare}) follows.\end{proof}

We shall need the following lemma in the proofs of Propositions \ref{lemK:indXplus1} and \ref{lem:indXplus1}.

	\begin{lemma}\label{lemK:ijmultisubsetnew}
		Let $J$ be a multi-subset of $[1, n]$ satisfying the following conditions:
		\begin{itemize}
		\item[($\ast$)] There exists a number $k \in J$ with multiplicity at least $2$ and every number $j$ with $k < j \le n$ is contained in $J$ with multiplicity $1$.
		\end{itemize}
	Then, we have
		\begin{equation}\label{eq:lemK:ijmultisubsetnew}
			f(J) \equiv 0 \mod I(\bar X) \,\,\text{ and }\,\, e(J) \equiv 0 \mod 2.
		\end{equation}
	\end{lemma}
	\begin{proof}
Let $k\in J$ denote a number with multiplicity at least $2$ as in $(\ast)$. We prove by decreasing induction on $k$. If $2k>n$, then the first equation in (\ref{eq:lemK:ijmultisubsetnew}) follows directly from 
(\ref{krelationinbarmodi}) since $f(2k)= 0$. Otherwise, by (\ref{krelationinbarmodi})
again, we have $f(J) \equiv f(J') \mod I(\bar X)$, where $J' = J+\{2k\}-\{k,k\}$. Since $2k\in J'$ has multiplicity $2$ and every $j$ with $2k<j\leq n$ is contained in $J'$, it follows by the induction that $f(J')\equiv 0 \mod I(\bar{X})$, whence the first equation follows. The second equation in (\ref{eq:lemK:ijmultisubsetnew}) follows from the first one by applying $\psi$ in (\ref{morphismpsi}) together with 
(\ref{imageunderphibar}).
	\end{proof}

We denote
\begin{equation*}
I_{4}=[6, \tfrac{n }{4}+1]\backslash \{2^{i}\,|\, 3\leq i\leq v(n)-2\}.
\end{equation*}

Now we will prove the main result of this section, which plays a key role in the proof of $2$-divisibility  of $\phi(x)$ (see Proposition \ref{prop:GK}).
For $n=8$, an analogue of the following proposition is proved inside the proof of \cite[Proposition 4.4]{Kar2020} 
% spead inside the proof
(see Remark \ref{lemK:indXplus1n8} below).

\begin{prop}\label{lemK:indXplus1}
	Let $n\geq 16$. Then, the following equations hold modulo $I(\bar X)^{v(\ind X)+1}$.
\begin{enumerate}
	\item[($a$)] 	$f(1)^{\frac{n^{2}}{4}-1}\cdot g(I_{3}\cup I_{4}\cup \{2,3\}) 
	\equiv 
	0$
	
	\smallskip

	\item[($b$)] 	$f(1)^{\frac{n^{2}}{4}-2}\cdot g(I_{3}\cup I_{4}\cup \{2,4\})\equiv 2^{v(\ind X)-2} t^2\cdot  f([2,n])$.
\end{enumerate}
\end{prop}
\begin{proof}
By (\ref{krelationinbarmodi}),
$f(1)^{m}\equiv f(m) \mod I(\bar{X})$ for any $2$-power integer $m$, thus 
\begin{equation}\label{eq:fonenequiv}
f(1)^{n-2}\equiv \prod_{k=1}^{v(n)-1}f(2^{k}) \,\,\text{ and }\,\, f(2^{i})f([\tfrac{n}{2}, n]) \prod_{k=1}^{v(n)-2}f(2^{k})\equiv 0 \mod I(\bar{X})
\end{equation}
for any $1\leq i\leq v(n)$. Since
\begin{equation}\label{eq:frac3n4}
	\tfrac{3n}{4}-v(n)+1 +|I_{4}| = n-2v(n)+1,
\end{equation}by Corollary \ref{cor:fonepower} and the first equation in (\ref{eq:fonenequiv}), the following equation holds modulo 
$I(\bar X)^{n-2v(n)+1}$: 
\begin{equation*}
	f(1)^{\frac{n^{2}}{4}-2}g(I_{3}\cup I_{4})\equiv 2^{\frac{n}{2}-v(n)+2}\cdot g([\tfrac{n}{4}+2, \tfrac{n}{2}-1]\cup I_{4}) f([\tfrac{n}{2}, n]) \prod_{k=1}^{v(n)-2}f(2^{k}).
\end{equation*}
Let $A$ denote the right hand side of the preceding equation without the factor $2^{\frac{n}{2}-v(n)+2}$. Thus, to prove $(a)$ and $(b)$, by (\ref{torsionindex}) it suffices to show the following congruences modulo $I(\bar X)^{\frac{n}{2}-v(n)+1}$:
\begin{align*}
\big(2f(2)-tf(3)\big)\big(2f(3)-tf(4)\big)A&\equiv 0,\\ \big(2f(2)-tf(3)\big)\big(2f(4)-tf(5)\big)A&\equiv 2^{\frac{n}{2}-v(n)-2} t^2 f([2,n]),
\end{align*}
respectively.

Since $|[\tfrac{n}{4}+2, \tfrac{n}{2}-1]|=\tfrac{n}{4}-2$ and $|I_{4}|=\tfrac{n}{4}-v(n)$, we have
\begin{equation}\label{gnoverfour}
	g([\tfrac{n}{4}+2, \tfrac{n}{2}-1]\cup I_{4})\in I(\bar X)^{\frac{n}{2}-v(n)-2},	
\end{equation}
thus by the second equation in (\ref{eq:fonenequiv}) and $f(3)^{2}\equiv f(6) \mod I(\bar{X})$ it is enough to show that the following hold modulo 
$I(\bar X)^{\frac{n}{2}-v(n)-1}$:
\begin{equation*}
		f(6)\cdot A\equiv 0 \,\,\text{ and }\,\,f(\{3,5\})\cdot A\equiv 2^{\frac{n}{2}-v(n)-2} f([2,n]),
\end{equation*}
respectively. By (\ref{gnoverfour}), 
the left-hand sides of these congruences
can be expanded as 
\begin{equation*}
	\sum_{J}a(J)f(J)\,\, \text{ and }\,\, 2^{\frac{n}{2}-v(n)-2} f([2,n])+\sum_{L\neq [2,n]}b(L)f(L),
\end{equation*}
respectively, where 
$a(J),\, b(L)\in I(\bar X)^{\frac{n}{2}-v(n)-2}$, 
$J$ denotes a multi-subset of $\{2,4\}\cup [6,n]$, and $L$ denotes a multi-subset of $[2,n]$. Since each of $J$ and $L$
satisfies the condition $(\ast)$ in Lemma \ref{lemK:ijmultisubsetnew}, the statement follows.
\end{proof}
\begin{rem}\label{lemK:indXplus1n8}
For $n=8$, the congruences $(a)$ and $(b)$ in Proposition \ref{lemK:indXplus1} still hold if $I_{3}=[5,7]$ is replaced by $I_{3}'=[6,7]$, i.e.,
\begin{align}
f(1)^{15}\cdot g([6,7]\cup \{2,3\})& \equiv 
0 \mod I(\bar{X})^{5},\label{eq:fonefifteen1}\\ 
f(1)^{14}\cdot g([6,7]\cup \{2,4\})& \equiv 
2^{2}t^{2}f([2,8]) \mod I(\bar{X})^{5}.\label{eq:fonefifteen2}
\end{align}

Indeed, since $f(1)^{8}\equiv f(8)$ and $f(8)^{2}\equiv f(7)^{2}\equiv 0 \mod I(\bar{X})$, 
\begin{equation*}
f(1)^{8}g([6,7])\equiv 2^{2}f([6,8]) \mod I(\bar{X})^{3}.
\end{equation*}
Hence, the congruences (\ref{eq:fonefifteen1}) and (\ref{eq:fonefifteen2}) follow from 
\begin{equation*}
	f(2)^2 f(4) f(8) \equiv f(4)^2 f(8) \equiv f(8)^2 \equiv f(3)^2 f(6) \equiv f(6)^2 \equiv 0 \mod I(\bar X).
\end{equation*}
\end{rem}

\smallskip

Below we provide analogues of Corollary \ref{cor:fonepower} and Proposition \ref{lemK:indXplus1}(a) in $\CH(\bar{X})$. In the analogue (Proposition \ref{lem:indXplus1}) of Proposition \ref{lemK:indXplus1}(a), the element of $\CH(\bar X)$ is not divisible by the same power of $2$ as the power of 
$I(\bar X)$ in Proposition \ref{lemK:indXplus1}(a) itself anymore. This difference will enable us to prove that the element $x$ in (\ref{eq:elementx}) is not divisible by $2$ in $\CH(X)$. For each $i \in [1,n]$, we denote 
\begin{equation}\label{eq:steenrodshortenedformula}
	\widehat{S}^j(i) = \binom{i-1}{j} c(i+j),
	\quad
	\widehat{S}(i) = \sum_{j=0}^{i-1} \binom{i-1}{j} c (i+j).
\end{equation}
In other words, $\widehat{S}(i)$ is an integral representative of the sub-linear combination of $S(\bar{c}(i))$ that consists of multiples of single $c(j)$s only, not of their products.
For a subset $L \subset [1,n]$,
denote $\widehat{S}(L) = \prod_{l \in L} \widehat{S}(l)$.

\begin{cor}\label{lem:eonepower}
	For any $n\geq 8$, 
	we have $e(1)^{\frac{n^{2}}{4}-n}\cdot \res(\widehat{S}(I_{3}))\equiv 0 \mod 2^{\frac{3n}{4}-v(n)}$ and 
	\begin{equation*}
		e(1)^{\frac{n^{2}}{4}-n}\cdot \res(\widehat{S}(I_{3}))\cdot e(\tfrac{n}{2})\equiv 
		2^{\frac{3n}{4}-v(n)}\cdot e([\tfrac{n}{4}+2, n])  \mod 2^{\frac{3n}{4}-v(n)+1}.
	\end{equation*}	
\end{cor}
\begin{proof}
	The proof is similar to the proof of Corollary \ref{cor:fonepower}.
	Again, the first statement follows immediately from the last statement of Lemma \ref{lemK:eonenj}.
	For the second statement,
	we show that
	\begin{equation}\label{eq:eonensquare}
		e(1)^{\frac{n^{2}}{4}-n} \res(\widehat{S}(I_{3}))\equiv e(1)^{\frac{n^{2}}{4}-n}\cdot 2^{\frac{n}{4}+1}\cdot e(I_{3}) \mod 2^{\frac{3n}{4}-v(n)+1}.
	\end{equation}
	Then, the second statement immediately follows by Proposition \ref{propK:eoneeIthree} and (\ref{eq:vnoverfour}).
	
	The term $\widehat{S}(I_{3})$ is expanded as 
	\begin{equation*}
	\widehat{S}(I_{3})=c(I_{3})+\sum_{L}a(L)c(L),
	\end{equation*}
where $a(L)\in \mathbb{N}$ 
and the sum ranges over some multi-subsets $L \subset \bar{I}_{3}$ 
such that $L$ contains either $n$ or a multiple element. Then, by (\ref{eq:rescitwoe}) we get
\begin{equation}\label{eq:reshatSI3}
	\res(\widehat{S}(I_{3}))=2^{\frac{n}{4}+1}e(I_{3})+2^{\frac{n}{4}+1}\sum_{L}a(L)e(L).
\end{equation}
	
If $n\in L$, 
then since $e(n)\equiv e(1)^{n} \mod 2$ by (\ref{relationinbar}), we get by Lemma \ref{lemK:eonenj} 
\begin{equation}\label{eq:ninKeone}
e(1)^{\frac{n^{2}}{4}-n} \cdot 2^{\frac{n}{4}+1} e(L) \equiv e(1)^{\frac{n^{2}}{4}} \cdot 2^{\frac{n}{4}+1} e(L-\{n\}) \equiv 0 \mod 2^{\frac{3n}{4} -v(n)+1}.
\end{equation}
If $L$ has a multiple element $i\in \bar{I}_{3}$, then $e(i)^{2}\equiv 0 \mod 2$ by (\ref{relationinbar}), thus by Lemma \ref{lemK:eonenj} again, we get
\begin{equation}\label{eq:Khasmultiple}
			e(1)^{\frac{n^{2}}{4}-n} \cdot 2^{\frac{n}{4}+1} e(L) = e(1)^{\frac{n^{2}}{4}-n} \cdot 2^{\frac{n}{4}+1} e(L-\{i,i\}) e(i)^2 \equiv 0 \mod 2^{\frac{3n}{4} -v(n)+1}.
\end{equation}
Hence, the equation in (\ref{eq:eonensquare}) immediately follows from (\ref{eq:reshatSI3}), (\ref{eq:ninKeone}), and (\ref{eq:Khasmultiple}).
\end{proof}

For $n=8$, an analogue of the following proposition is proved inside the proof of \cite[Proposition 3.3]{Kar2020} 
% no precise ref
(see Remark \ref{lem:indXplus1n8} below).

\begin{prop}\label{lem:indXplus1}
Let $n\ge 16$. Then, we have
	\begin{equation*}
		e(1)^{\frac{n^{2}}{4}-1}\cdot \res(\widehat{S}(I_{3}\cup I_{4}\cup \{2,3\}))
		\equiv 
		\ind X\cdot e([1,n]) \mod 2 \ind X.
	\end{equation*}
\end{prop}
\begin{proof}
	The proof is similar to the proof of Proposition \ref{lemK:indXplus1}. By (\ref{relationinbar}), $e(1)^{m}\equiv e(m) \mod 2$ for any $2$-power integer $m$, thus 
	\begin{equation}\label{eq:eonenequiv}
		e(1)^{n-1}\equiv \prod_{k=0}^{v(n)-1}e(2^{k}) \,\,\text{ and }\,\, 
		e(2^{i})e(1)e(2)e(4)e([6, n]) \equiv 0 \mod 2
	\end{equation}
	for any $0\leq i\leq v(n)$. By Corollary \ref{lem:eonepower}, (\ref{eq:frac3n4}), and the first equation in (\ref{eq:eonenequiv}) we get a congruence modulo
	$2^{n-2v(n)+1}$:
	\begin{equation}\label{eq:eonefracnsqures}
		e(1)^{\frac{n^{2}}{4}-1}\res(\widehat{S}(I_{3}\cup I_{4})) \equiv 
		2^{\frac{3n}{4}-v(n)}\cdot \res(\widehat{S}(I_{4}))\cdot e([\tfrac{n}{4}+2, n]) \prod_{k=0}^{v(n)-2}e(2^{k}).
	\end{equation}
	
The term $\widehat{S}(I_4)$ is expanded as
\begin{equation*}
\widehat{S}(I_4)=c(I_{4})+\sum_{L}a(L)c(L),
\end{equation*}
where $a(L)\in \mathbb{N}$ and the sum ranges over some multi-subsets $L\subset [1,n]$ 
such that the multiset 
$L+ [\tfrac{n}{4}+2, n] + \{2^{i}\,|\, 3\leq i\leq v(n)-2\}$ 
satisfies the condition $(\ast)$ in Lemma \ref{lemK:ijmultisubsetnew}. Since $|I_4| = \tfrac{n}{4}-v(n)$, by (\ref{eq:rescitwoe}) we get
\begin{equation}\label{eq:reshatSI4}
\res(\widehat{S}(I_{4}))=2^{\frac{n}{4}-v(n)}e(I_{4})+2^{\frac{n}{4}-v(n)}\sum_{L}a(L)e(L).
\end{equation}
By Lemma \ref{lemK:ijmultisubsetnew}, $e(L) e([\tfrac{n}{4}+2, n]) \prod_{k=3}^{v(n)-2}e(2^{k})$ 
is divisible by 2, thus by (\ref{eq:eonefracnsqures}) and (\ref{eq:reshatSI4}) we have a congruence modulo
$2^{n-2v(n)+1}$
\begin{equation*}
	e(1)^{\frac{n^{2}}{4}-1}\res(\widehat{S}(I_{3}\cup I_{4})) \equiv
2^{n-2v(n)} e(1)e(2)e(4)e([6,n]).
\end{equation*}

Now let us multiply both sides of the last congruence by
\begin{equation}\label{eq:shat23}
\res(\widehat{S}(\{2,3\}))=2^{2}(e(2)+e(3))(e(3)+2e(4)+e(5)).
\end{equation} 
Since $2^{n-2v(n)+3} = 2 \ind X$, 
by the second equation in (\ref{eq:eonenequiv}) it is enough to show that the following holds modulo $2$
\begin{equation*}
e(3)^2 \cdot e(1)e(2)e(4)e([6,n]) \equiv 0.
\end{equation*}
Since $e(3)^2 \equiv e(6) \mod 2$ by (\ref{relationinbar}), and the multiset $\{6\} + [6,n]$ satisfies the condition $(\ast)$ in Lemma \ref{lemK:ijmultisubsetnew}, the term $e(\{6\} + [6,n])$ is divisible by $2$, which completes the proof.\end{proof}
\begin{rem}\label{lem:indXplus1n8}
		Let $n=8$ and $I_{3}'=[6,7]$. Then, the statement of Proposition \ref{lem:indXplus1} becomes
	\begin{equation}\label{eq:eonefifteen}
	e(1)^{15} \cdot \res(\widehat{S}(I_{3}' \cup \{2,3\})) \equiv 2^4 \cdot e([1,8]) \mod 2^5.
	\end{equation}
	Since  
	$e(1)^{8}\equiv e(8)$,\,\,\, $e(8)^{2}\equiv e(7)^{2}\equiv 0 \mod 2$, and 
	\begin{equation*}
	\widehat{S}(6) = c(6)+5c(7)+10c(8),\,\,\, \widehat{S}(7) = c(7)+6c(8),
	\end{equation*}
    we have
	\begin{equation*}
	e(1)^{8}\res(\widehat{S}(I_{3}')) \equiv 2^{2}e([6,8]) \mod 2^{3}.
	\end{equation*}
	Hence, the congruence (\ref{eq:eonefifteen}) follows from (\ref{eq:shat23}) and
	\begin{equation*}
	e(2)^2 e(4) e(8) \equiv e(4)^2 e(8) \equiv e(8)^2 \equiv e(3)^2 e(6) \equiv e(6)^2 \equiv 0 \mod 2.
	\end{equation*}
\end{rem}

\section{Proof of Theorem \ref{mainthm}}\label{sec:proofofthm}
In this section, we set
\begin{equation}\label{eq:setJcases}
J=\begin{cases}
[2,3]\cup I_{3}\cup I_{4} & \text{ if } n\geq 16,\\
[2, 3]\cup I'_{3}\cup I_{4}=\{2,3, 6,7\} &\text{ if } n=8.
\end{cases}
\end{equation}
Then, a direct computation shows that 
\begin{equation}\label{eq:sizeofJset}
|J| = \tfrac{n}2 - v(n)+3.
\end{equation}

Consider the following element
\begin{equation}\label{eq:elementx}
	x=e^{\frac{n^{2}}{4}-1}\cdot \prod_{j \in J}c(j)\in \CH(X).
\end{equation}
Since $\dim X=\tfrac{n^{2}+n}{2}$ and $(\tfrac{n^{2}}{4}-1)+\sum_{j\in J}j=\dim X- 3$, we have $x\in \CH_{3}(X)$ and $\phi(x)\in K(X)^{(\dim X -3)}$. We first prove the $2$-divisibility of the image of $x$ under the map in (\ref{eq:conjecture}).

\begin{prop}\label{prop:GK}
Let $y = f(1)^{\frac{n^{2}}{4}-1}\cdot g(J)$, where $J$ denotes the set in $(\ref{eq:setJcases})$. Then, $y \in I(X)$ and $\varphi(x)$ is divisible by 2 in $GK(X)$.
\end{prop}
\begin{proof}
Let $z=f(1)^{\frac{n^{2}}{4}-2}\cdot g(J')$, where $J'$ denotes the set obtained from $J$ by replacing the element $3$ with $4$. Then, by (\ref{eq:gikigicki}) both $y$ and $z$ belong to $\widetilde{K}^{\dim X -3}(X)$. We first show that $y\in I(X)$. By Proposition \ref{lemK:indXplus1} (a) for $n \geq 16$ and by (\ref{eq:fonefifteen1}) for $n=8$, we can write
\begin{equation*}
	y=2^{m+1}\cdot y_{0}+2^{m}t\cdot y_{1}+2^{m-1}t^{2}\cdot y_{2}+2^{m-2}t^{3}\cdot y_{3},
\end{equation*}
where $m=v(\ind X)$ and $y_{i}\in \widetilde K^{\dim X-3+i} (\bar X)$. By (\ref{eq:twoindextindex}), it suffices to prove that
\begin{equation*}
	y':=2^{m-1}t^{2}\cdot y_{2}+2^{m-2}t^{3}\cdot y_{3} \in I(X).
\end{equation*}

We simply write $\pk$ and $\lk$ for the classes of $\prod_{i=1}^{n}\ekbar(i)$ and $\prod_{i=2}^{n}\ekbar(i)$ in $K(\bar{X})^{(\dim\bar{X})}$ and $K(\bar{X})^{(\dim\bar{X}-1)}$, respectively, as in Section \ref{subsec:Grothendieck}. Since $2^{m-1}t^{2}\cdot t\pk=2^{m-2}t^{3}\cdot 2\pk$, by (\ref{Kpl}), 
$y'$
can be written as
\begin{equation}\label{eq:ybiga2}
	y'=a(2^{m-1}t^{2})\cdot \lk u^{\dim X - 1} + b(2^{m-2}
	t^{3}
	)\cdot \pk u^{\dim X}
\end{equation}
for some $a, b\in \Z$. On the other hand, by Proposition \ref{lemK:indXplus1} (b) for $n \geq 16$ and by (\ref{eq:fonefifteen2}) for $n=8$, we have
\begin{equation}
2z\equiv (2^{m-1}t^{2}\cdot \lk )u^{\dim X -1}\,\, \text{ and }\,\, tf(1)z\equiv (2^{m-2}
t^{3}
\cdot \pk)
u^{\dim X}
\mod I(\bar{X})^{m+2},
\end{equation}
thus it follows by (\ref{eq:ybiga2}) that
\begin{equation*}
y'-2az-btf(1)z\in I(\bar{X})^{m+2}\cap \widetilde{K}^{\dim X -3}
(X).
\end{equation*}
Hence, we can write
\begin{equation*}
	y'-2az-btf(1)z=2^{m+2}\cdot z_{0}+2^{m+1}t\cdot z_{1}+2^{m}t^{2}\cdot z_{2}+2^{m-1}t^{3}\cdot z_{3},
\end{equation*}
where $z_{i}\in \widetilde K^{\dim X-3+i} (\bar X)$, thus by (\ref{eq:twoindextindex}), it suffices to prove that
\begin{equation*}
	y'':=2^{m-1}t^{3}\cdot z_{3} \in I(X).
\end{equation*}

By (\ref{Kpl}), $y''$ can be written as 
\begin{equation}\label{eq:y2mminusone}
	y''=b'(2^{m-1}t^{3})\cdot \pk u^{\dim X}
\end{equation}
for some $b'\in \mathbb{Z}$. Since
\begin{equation*}
	(2^{m-1}t^{3}\cdot \pk)u^{\dim X}\equiv 2tf(1)z \mod I(\bar{X})^{m+3},
\end{equation*}
we obtain
\begin{equation*}
y''-2b'tf(1)z \in I(\bar{X})^{m+3}\cap \widetilde{K}^{\dim X -3}(X).
\end{equation*}
As every element in $I(\bar{X})^{m+3}\cap \widetilde{K}^{\dim X -3}(X)$ belongs to $I(X)$ by (\ref{eq:twoindextindex}), we get $y'' \in I(X)$, and therefore $y \in I(X)$.

For the statement about $x$, note that by the second equation in (\ref{eq:gikigicki}), we have 
\begin{equation*}
f(1)^{\frac{n^{2}}{4}-1}\cdot \prod_{j \in J}\ck(j)u^{j} \in I(X).
\end{equation*}
Also, by (\ref{imageunderphi}) and Lemma \ref{ecliffordtogkxnobar}, we get 
\begin{equation*}
\varphi(x) = \xi \big( f(1)^{\frac{n^{2}}{4}-1}\cdot \prod_{j \in J}\ck(j)u^{j} \big),
\end{equation*}
where $\xi$ denotes the morphism in (\ref{eq:gkalternativeconstruction}). Hence, by (\ref{eq:xiix2gkx}) $\varphi(x)$ is divisible by $2$.
\end{proof}
\begin{rem}\label{rem:proofKpart}
In the proof of Proposition \ref{prop:GK}, we have shown that $(2^{m-1}\lk)u^{\dim X-3}$ and $(2^{m-2}\pk)u^{\dim X-3}$ are contained in $I(X) +I(\bar{X})^{m+2}$. 
In fact, we could alternatively show that the following slightly stronger statement holds:
\begin{equation}\label{eq:strongeremone}
	(2^{m-1}\lk)u^{\dim X-3},\, (2^{m-2}\pk)u^{\dim X-3}\in I(X).
\end{equation}
Since $z-(2^{m-2}\lk)u^{\dim X-3}\in I(\bar{X})^{m+1}\cap \widetilde{K}^{\dim X-3}(\bar{X})$ (Proposition \ref{lemK:indXplus1} (b)), the same argument as in the beginning of the proof of Proposition \ref{prop:GK} shows that
\begin{equation}\label{eq:zminustwo}
z-(2^{m-2}\lk)u^{\dim X-3}\equiv (a2^{m-1}\lk+b2^{m-2}\pk) u^{\dim X - 3} \mod I(X)
\end{equation}
for some $a, b\in \Z$. Since $2z,\, 
\ekbar(1)z
\in I(X)$, multiplying the congruence in (\ref{eq:zminustwo}) by $2$ and 
$\ekbar(1)$, 
we have
\begin{equation*}
(2^{m-1}\lk+b2^{m-1}\pk)u^{\dim X - 3},\,\, (2^{m-2}\pk+a2^{m-1}\pk)u^{\dim X - 3} \in I(X).
\end{equation*}
As $(2^{m-1}\pk)u^{\dim X - 3}=\ekbar(1)
(2^{m-1}\lk+b2^{m-1}\pk)u^{\dim X - 3}\in I(X)$, the statement in (\ref{eq:strongeremone}) follows.

\end{rem}

Recall from (\ref{eq:steenrodformula}) that $Q(i,j)$ 
denotes a linear combination   of $\bar{c}(k)\bar{c}(i+j-k)$ for $1\leq k\leq i$. We write $\widehat{Q}(i,j)$ 
for the linear combination obtained from $Q(i,j)$ 
by replacing every term $\bar{c}(k)\bar{c}(i+j-k)$ with $c(k)c(i+j-k)$. Set
\begin{equation*}
\widetilde{S}^j(i)=\widehat{S}^j(i)+\widehat{Q}(i,j),
\quad \widetilde{S}(i)=\sum_{j\geq 0} \widetilde{S}^j(i),
\end{equation*}
where $\widehat{S}^j(i)$ denotes the term in (\ref{eq:steenrodshortenedformula}), i.e., $\widetilde{S}(i)$ is an integral representative of $S(\bar{c}(i))$. For a subset $L \subset [1,n]$,  
denote $\widetilde{S}(L) = \prod_{l \in L} \widetilde{S}(l)$.

	Before we prove that $x \in \CH(X)$ is not divisible by $2$, we shall need the following lemma.

\begin{lemma}\label{lem:eonenjtorsioncor}
	For any $n\geq 8$, we have 
	$e(1)^{\frac{n^{2}}{4}} \res(\widetilde{S}(J)) \equiv 0 \mod 2 \ind X$ 
	and
\begin{equation*}
e(1)^{\frac{n^{2}}{4}-1} \res(\widetilde{S}(J))\equiv e(1)^{\frac{n^{2}}{4}-1}\res(\widehat{S}(J)) \mod 2 \ind X.
\end{equation*}
\end{lemma}
\begin{proof}
Note that $v(\ind X)=n - 2v(n)+2$. 
By Lemma \ref{lemK:eonenj} and (\ref{eq:sizeofJset}), we have
\begin{equation*}
	v\left( 2^{|J|} e(1)^{\frac{n^{2}}{4}} \right) \ge (\tfrac{n}{2}-v(n)+3) + (\tfrac{n}{2}-2) >
	v(\ind X) 
	\text{ and }
\end{equation*}
\begin{equation*}
v\left( 2^{|J|+1} e(1)^{\frac{n^{2}}{4}-1} \right) \ge 
v\left( 2^{|J|+1} e(1)^{\frac{n^{2}}{4}-n} \right) \ge (\tfrac{n}{2}-v(n)+4) + (\tfrac{n}{2}-1 - v(n)) > v(\ind X).
\end{equation*}
Hence, the first statement
follows from (\ref{eq:rescitwoe}). For the second statement, note additionally that $\widetilde{S}(J) - \widehat{S}(J)$ 
is the sum of several products of the form $\prod_{1 \le k \le |J|} A_k$, 
where each $A_k$ can be either $\widehat S(i)$, or $\widehat{Q}(i,j)$, and at least one factor $\widehat{Q}(i,j)$ is present.
So, by (\ref{eq:rescitwoe}), $\res(\prod_{1 \le k \le |J|} A_k)$ is divisible by $2^{|J|+1}$, and the second statement follows.
\end{proof}

Finally, let us prove the non-2-divisibility in $\CH(X)$.

\begin{prop}\label{prop:chow}
	For any $n\geq 8$,
	the element $x$ in $(\ref{eq:elementx})$ is not divisible by $2$.
\end{prop}
\begin{proof}
Let $w=(e+e^{2})^{\frac{n^{2}}{4}-1} \widetilde{S}(J) $. Then, the ($\dim X$)th degree homogeneous part of $w$ is an integral representative of $S^{3}(\bar{x})$, i.e., the ($\dim X$)th degree homogeneous part of $\bar{w}\in \Ch(X)$ is equal to $S^{3}(\bar{x})$. We show that
	\begin{equation}\label{eq:resS0x}
	\res (w)
	\equiv \ind X \cdot p \mod 2 \ind X,
	\end{equation}
	where $p$ denotes the class of a rational point as in Section \ref{subsection:chow}. Since
	\begin{equation*}
	w=\big(e^{\frac{n^{2}}{4}-1} + e^{\frac{n^{2}}{4}} \alpha(e)\big) \widetilde{S}(J),
	\end{equation*}
	where $\alpha(e)$ is a polynomial in $e$ with integer coefficients, by Lemma \ref{lem:eonenjtorsioncor}, we have modulo $2 \ind X$:
		\begin{equation*}
	\res(w) \equiv e(1)^{\frac{n^{2}}{4}-1} \res(\widetilde{S}(J)) \equiv e(1)^{\frac{n^{2}}{4}-1} \res(\widehat{S}(J)).
	\end{equation*}
	Hence, by Proposition \ref{lem:indXplus1} we obtain (\ref{eq:resS0x}).

	Let $\deg:\CH(X)\to \Z$ denote the degree homomorphism. Then, it induces the morphism
	\begin{equation*}
	\tfrac{\deg}{\ind X}:\Ch(X)\to \Z/2\Z
	\end{equation*}   sending the class of a closed point $v$ of $X$ to the class of 
	$\deg(v)/\ind X$. 
	Since the restriction map commutes with the degree homomorphism, by (\ref{eq:resS0x}) we have
	\begin{equation*}
	\tfrac{\deg}{\ind X}(\bar{w})=\tfrac{\deg}{\ind X}(S^{3}(\bar{x}))=1.
	\end{equation*}
	Therefore, $\bar{x}$ is nonzero in $\Ch(X)$, thus $x$ is not divisible by $2$ in $\CH(X)$.\end{proof}

\begin{thm}
$\varphi$ is not injective.
\end{thm}
\begin{proof}
Follows from Proposition \ref{prop:GK}, Proposition \ref{prop:chow}, and the surjectivity of $\varphi$.
\end{proof}

\appendix
\section{Pieri formula in the Grothendieck ring of $\bar{X}$}

In this section, we give a proof of the congruence relations in (\ref{krelationinbar}). Using the Pieri-type formula in Lemma \ref{pieriformulaktheory}, we first compute the the products $\ekbar_{i}\ekbar_{m}$ in terms of the Schubert classes (Lemmas  \ref{ktheorysquaresbarek} and \ref{ktheoryproductsbarek}). Then, we derive the formulas for the square of $f(i)\in \widetilde K^{i}(\bar X)$ in Proposition \ref{ktheorysquaresckproducts}.

Recall that the group $K(\bar{X})$ is free abelian with basis the set of all products $\prod_{i\in [1, n]}\ekbar(i)$ (including the empty product, the unit). Recall also that a \emph{strict partition} in $[1,n]$ is a sequence $\lambda=(\lambda_{1},\ldots, \lambda_{m})$ such that $n\geq \lambda_{1}>\cdots >\lambda_{m}\geq 1$. The size of $\lambda$ is denoted by $|\lambda|=\lambda_{1}+\cdots+\lambda_{m}$. Then, the group $K(\bar{X})$ has another basis $\ekbar_{\lambda}\in K(\bar{X})^{(|\lambda|)}$, where $\lambda$ ranges over all strict partitions in $[1,n]$ including the empty partition, given by the Schubert classes. Note that if $\lambda$ consists of a single element $\{i\}$, then $\ekbar_{i}=\ekbar(i)$. We allow notation $\ekbar_\lambda$ with $\lambda$ an arbitrary finite decreasing sequence of natural numbers: if $\lambda$ contains numbers bigger than $n$, we set $\ekbar_{\lambda} =0$.

We shall first recall some basic notions from \cite{buchravikumar}. Let $\lambda$ be a finite decreasing sequence of natural numbers. The \emph{shifted diagram} of $\lambda$ is an array of boxes in which the $i$th row has $\lambda_{i}$ boxes, and is shifted $i-1$ units to the right with respect to the top row. We denote the number of rows of $\lambda$ by $l(\lambda)$. A \emph{skew shifted diagram} (or shape) $\nu /\lambda$ is obtained by removing a shifted diagram $\lambda$ from a larger shifted diagram $\nu$ containing $\lambda$. The number of boxes in $\nu /\lambda$ is denoted by $|\nu /\lambda|$. A skew shifted diagram is called \emph{connected} if all boxes share an edge. A skew shifted diagram is called a \emph{rim} if it does not contain a pair of boxes one of which is located strictly to the right (east) and strictly to the bottom (south) of the other one.

\begin{dfn}\label{dfn:KOG}\cite[Section 4]{buchravikumar}
	Let $\theta$ be a rim. A \emph{KOG-tableau} of $\theta$ is a labeling of the boxes of $\theta$ with positive
	integers such that 
	\begin{enumerate}
		\item[(i)] each row (resp. column) of $\theta$ is strictly increasing from left (resp. top) to right (resp. bottom); and
		
		\item[(ii)] each box is either smaller than or
		equal to all the boxes south-west of it, or it is greater than or equal to all the boxes 
		south-west of it. 
	\end{enumerate}
	If $\theta$ is not a rim, then there are no KOG-tableaux with shape $\theta$. The \emph{content} of a
	KOG-tableau is the set of integers contained in its boxes.	
\end{dfn}
\begin{rem}\label{koggreaterorless}
	Let $B$ be a box in a KOG-tableau of shape $\theta$. If there is a box in $\theta$ located directly to the left of $B$, then 
	$B$ is actually greater than or equal to all the boxes 
	south-west of it.
	If there is a box in $\theta$ directly below $B$, then 
	$B$ is actually less than or equal to all the boxes 
	south-west of it.
\end{rem}

\begin{example}\label{ex:KOGT}
	(1) Let us consider the following rim with two rows
	\begin{equation}\label{ex:table1}
	\begin{ytableau}
	\none & \none & \none & \none[\cdots] & a_{1} \\
	b_{1} & \cdots  & b_{r} & \none & \none
	\end{ytableau}
	\end{equation}
	such that the two rows of the rim are disconnected\footnote{ Here and further, dots outside boxes denote empty space of any nonnegative length, in particular, length $0$ is possible. In other words, 
	it is possible that in the tableau (\ref{ex:table1}), the cells $a_1$ and $b_{r}$ share a vertex (but not an edge).}, where the top row consists of only one box and the bottom row consists of $r$ boxes. Then, for any $r\geq 1$, the number of KOG-tableaux of shape (\ref{ex:table1}) with content $[1,r+1]$ is equal to $2$. This can be verified in the following way. As the number of boxes of (\ref{ex:table1}) is equal to $r+1$, $a_{1}, b_{1},\ldots, b_{r}$ are distinct numbers of $[1,r+1]$. If $a_{1}>b_{r}$, then by Definition \ref{dfn:KOG} (i) we have the unique KOG-tableau with labeling $(a_{1},b_{1},\ldots, b_{r})=(r+1,1,\ldots, r)$. Otherwise, by Definition \ref{dfn:KOG} (i), (ii) we see that $b_{r}>b_{r-1}>\cdots b_{1}>a_{1}$, thus we also have the unique KOG-tableau with labeling $(a_{1},b_{1},\ldots, b_{r})=(1,2,\ldots, r+1)$.
	
	\medskip
	
	(2) Now consider the following diagram, which is the same as the diagram above, but with one more cell added to the first row.
	\begin{equation}\label{ex:table2}
	\begin{ytableau}
	\none & \none & \none & \none[\cdots] & a_{1} & a_{2} \\
	b_{1} & \cdots  & b_{r} & \none & \none &\none
	\end{ytableau}
	\end{equation}
	Then, for any $r\geq 2$ the number of KOG-tableau of shape (\ref{ex:table2}) with content $[1,r+1]$ is equal to $3$: By Remark \ref{koggreaterorless}, we obviously get $a_{2}=r+1$. If $a_{1}\geq b_{r}$, then there is a unique KOG-tableau with labeling $(a_{1}, a_{2}, b_{1},\ldots, b_{r})=(r, r+1, 1,\ldots, r)$. Otherwise, we have $a_{1}\leq b_{1}<\cdots <b_{r}\leq r+1$, thus there are exactly two KOG-tableaux with labelings $(a_{1}, a_{2}, b_{1},\ldots, b_{r})=(1, r+1, 1,\ldots, r)$ and $(1, r+1, 2,\ldots, r+1)$.
\end{example}

We shall make use of the following combinatorial Pieri-type formula due to Buch and Ravikumar:

\begin{lemma}[{\cite[Corollary 4.8]{buchravikumar}}]\label{pieriformulaktheory}
	Let $1\leq i\leq n$ be an integer. For strict partitions $\lambda$ and $\nu$ in $[1,n]$, we denote by $C_{\lambda, i}^{\nu}$ the number of 
	KOG-tableaux of shape $\nu / \lambda$ with content $[1,i]$. Then,
	\begin{equation*}
	\ekbar_{i} \ekbar_\lambda = \sum_{\nu\, \subseteq\, [1,n]}
	(-1)^{\abs{\nu/ \lambda} - i}\cdot C_{\lambda, i}^{\nu}\cdot \ekbar_{\nu}.
	\end{equation*}
\end{lemma}
To be precise, in view of our convention that $\ekbar_{\nu}$ is defined and equals zero for any finite 
decreasing sequence of natural numbers $\nu$ containing numbers bigger than $n$, we will use this formula with the sum over ``strict partitions $\nu$''
replaced with the sum over ``decreasing sequences of natural numbers $\nu$''. All extra summands appearing this way are zeros, even if the coefficients 
$C_{\lambda, i}^{\nu}$ alone are not zeros.

Now we compute the coefficients (modulo terms in $K(\bar{X})^{(2i+2)}$) in the Pieri formula (Lemma \ref{pieriformulaktheory}) for $\lambda=(i)$.

\begin{lemma}\label{ktheorysquaresbarek}
We have $\ekbar_{1}^{2} = \ekbar_{2}$ and $\ekbar_{n}^{2} = 0$ in $K(\bar{X})$, and the following relations hold modulo $K (\bar{X})^{(2i+2)}:$
\begin{equation*}
	\ekbar_{i}^{2} \equiv
	\ekbar_{2i}+2(\sum_{k=1}^{i-1}\ekbar_{i+k,\, i-k})-\ekbar_{i+1,\,i}-3(\sum_{k=2}^{i-1}\ekbar_{i+k,\, i-k+1})-2\ekbar_{2i,\,1} 
\end{equation*}
for any $1<i<n$.
\end{lemma}
\begin{proof}
For now, in addition to $1 < i < n$, let us also allow $i=1$ and $i=n$. Let us use Lemma \ref{pieriformulaktheory} for this $i$ and for $\lambda = (i)$.
	First, note that if $l(\nu)\geq 3$, then the leftmost box of the third row of $\nu/ \lambda$ is strictly below and strictly to the right of the leftmost box of the second row of $\nu/ \lambda$, thus $\nu / \lambda$ is not a rim. Hence, we may assume that $\nu$ with $l(\nu)\leq 2$. 
	
	Since we consider the number of KOG-tableaux of shape $\nu / \lambda$ with content $[1,i]$, it suffices to consider $\nu$ with $\abs{\nu/ \lambda}\geq i$ (i.e.,  $\abs{\nu}\geq 2i$).
	
	If $l(\nu)=1$, then by Definition \ref{dfn:KOG} (i) $C_{\lambda, i}^{\nu}\neq 0$ if and only if $\abs{\nu / \lambda}=i$. In this case, $\nu/ \lambda$ is simply $\nu$ without the first leftmost $i$ boxes, thus $C_{\lambda, i}^{\nu}=1$, i.e., $\ekbar_{2i}$ occurs in $\ekbar_{i}^2$ with coefficient $1$.
	
	Let $i=1$. Then, again by Definition \ref{dfn:KOG} (i) $C_{\lambda, i}^{\nu}=0$ for any $\nu$ with $l(\nu)=2$, thus $\ekbar_2$ is the only summand in $\ekbar_1^2$.
	
	Let $i=n$. Since $\nu$ is a strict partition of $[1,n]$, the condition $l(\nu)=2$ implies that $|\nu|\leq 2n-1$. As $|\nu|\geq 2n$ and $\ekbar_{2n}=0$ by definition, we obtain the relation $\ekbar_{n}^{2}=0$.

	From now on, we assume that $2\leq i\leq n-1$, and we compute $\ekbar_{i}^{2}$
	modulo $K(\bar{X})^{(2i+2)}$. If $\abs{\nu/ \lambda}\geq i+2$, then $\abs{\nu}\geq 2i+2$, thus $\ekbar_{\nu}\in K(\bar{X})^{(2i+2)}$. 
	Therefore, we may assume that $\abs{\nu/ \lambda}=i$ or $i+1$.

	Let $\nu=(j,r)$ be such that $j>r$ and $j+r=2i$ or $j+r=2i+1$. By the definition of a rim, there cannot be more than one box in the top row of $\nu / \lambda$ 
	located directly above cells of the bottom row. Hence, it suffices to compute $C_{\lambda, i}^{\nu}$ for the following tableaux of shape $\nu/ \lambda$:
	\begin{equation}\label{table3}
		\begin{ytableau}
		\none  & \none    & a_{1} & \cdots &a_{k} \\
		b_{1}  & \cdots   & b_{r} 
		\end{ytableau}
		\text{\quad\quad or \quad\quad}
		\begin{ytableau}
		\none  & \none  & \none  & \none[\cdots] & a_{1} & \cdots &a_{k} \\
		b_{1}  & \cdots  & b_{r} & \none & \none & \none & \none ,
		\end{ytableau}
	\end{equation}
	where $k=j-i$, so $k=i-r$ or $k=i+1-r$ (note that the two rows of the second tableau are disconnected but they can share a vertex -- see Example \ref{ex:KOGT}).

	Assume that $k=i-r$. As $i +1 \le j \le 2i-1$, $1 \le r \le i-1$, we have $1 \le k \le i-1$. As $r<i$, two rows of the tableau (\ref{table3}) are disconnected. We show by induction that $C_{\lambda, k}^{\nu}=2$ for any $1\leq k \leq i-1$. The case $k=1$ follows from Example \ref{ex:KOGT} (1). Assume $i\geq 2$. Then, by Remark \ref{koggreaterorless} we have $a_{k}=i$, thus the statement follows by induction. Hence, for any $1\leq k\leq i-1$ the term $\ekbar_{i+k,\,i-k}$ occurs in $\ekbar_{i}^2$ with coefficient $2$.
	
	Now we assume that $k=i+1-r$. As $i +1 \le j \le 2i$, $1 \le r \le i$, we have $1 \le k \le i$. We shall consider three subcases: $k=1$, $k=i$, and $2\leq k\leq i-1$. If $k=1$, then $r=i$ and the first row of (\ref{table3}) consists of a single element $a_{1}$ just above $b_{r}$. Hence, by Remark \ref{koggreaterorless} we get $C_{\lambda, i}^{\nu}=1$ with a unique labeling $(a_{1},b_{1}, b_{2}\ldots, b_{r})=(1, 1, 2, \ldots, i)$, i.e., the term $\ekbar_{i+1,\,i}$ occurs in $\ekbar_{i}^2$ with coefficient $-1$. If $k=i$, then $r=1$ and two rows of the tableau (\ref{table3}) are disconnected. By definition \ref{dfn:KOG} (i), we see that $a_{m}=m$ for any $1\leq m\leq k$. By Remark \ref{koggreaterorless} applied to $a_{2}$, we have $b_{1}=1$ or $2$, thus $C_{\lambda, i}^{\nu}=2$, i.e., the term $\ekbar_{2i,\,1}$ occurs in $\ekbar_{i}^2$ with coefficient $-2$.

	Finally, let $2\leq k\leq i-1$. We show by induction that $C_{\lambda, i}^{\nu}=3$. The case $k=2$ immediately follows from Example \ref{ex:KOGT} (2). Assume that $k\ge 3$. Then, by Remark \ref{koggreaterorless}, $a_k = i$. If we remove the box $a_k$ from the diagram, 
	all numbers from $[1,i-1]$ must be present in the remaining boxes.
	But the content of the remaining boxes cannot be $[1,i]$ since otherwise we would have $a_{k-1} = a_{k}=i$ by Remark \ref{koggreaterorless}, 
	which contradicts Definition \ref{dfn:KOG} (i). So, we can proceed by induction on $k$ and get $C_{\lambda, i}^{\nu}=3$ for $2 \le k \le i-1$, i.e., for any $2\leq k\leq i-1$ the term $\ekbar_{i+k,\,i-k+1}$ occurs in $\ekbar_{i}^2$ with coefficient $-3$.
\end{proof}

	Recall that we have denoted $f(i) = \ekbar(i)u^{i}  = \ekbar_{i} u^{i}  \in \widetilde{K}(\bar{X})$. 
	We also simply denote $f_{m,\,i} = \ekbar_{m,\,i}\,u^{m+i}  \in \widetilde{K}(\bar{X})$. We deduce some formulas for $\widetilde{K}(\bar{X})$. The proof immediately follows from Lemma \ref{ktheorysquaresbarek}.

\begin{cor}\label{ktheorysquaresckcommas}
We have $f(1)^{2} = f(2)$ and $f(n)^{2} = 0$ in $\widetilde{K}(\bar{X})$, and the following relations hold modulo $I (\bar{X})^{2}:$
\begin{equation*}
f(i)^{2} \equiv f(2i)+2(\sum_{k=1}^{i-1} f_{i+k,\, i-k})-  t (\sum_{k=1}^{i-1}f_{i+k,\, i-k+1}) 
\end{equation*}
for any $1<i<n$.
\end{cor}

In the following, we compute the coefficients (modulo terms in $K(\bar{X})^{(i+m+2)}$) in the Pieri formula (Lemma \ref{pieriformulaktheory}) for any $\lambda=(m)$ with $m>i$.  

\begin{lemma}\label{ktheoryproductsbarek}
	Let $m>1$. Then, we have $\ekbar_{1}\ekbar_{m} = \ekbar_{m+1} + \ekbar_{m,\,1} - \ekbar_{m+1,\,1}$ in $K(\bar{X})$,
	and 
	the following relations hold modulo $K (\bar{X})^{(i+m+2)}:$
	\begin{equation*}
		\ekbar_{i} \ekbar_{m} \equiv
		\ekbar_{m+i}+\ekbar_{m,\, i}+2(\sum_{k=1}^{i-1}\ekbar_{m+k,\, i-k})-2 \ekbar_{m+1,\,i}-3(\sum_{k=2}^{i-1}\ekbar_{m+k,\, i-k+1})-2\ekbar_{m+i,\,1} 
	\end{equation*}
	for any $1<i<m$.
\end{lemma}
\begin{proof}
The proof is similar to the proof of Lemma \ref{ktheorysquaresbarek} above. For now, in addition to $1 < i < m$, let us also allow $i=1$. Let us use Lemma \ref{pieriformulaktheory} for this $i$ and for $\lambda = (m)$.
Then, the arguments of the first two paragraphs of the proof of Lemma \ref{ktheorysquaresbarek} show that $l(\nu) \le 2$ and $\abs{\nu / \lambda} \ge i$.

If $l(\nu)=1$, then it follows from Definition \ref{dfn:KOG} (i) that $C_{\lambda, i}^{\nu}=1$ if $\abs{\nu / \lambda} = i$, 
		and $C_{\lambda, i}^{\nu}=0$ 
		otherwise.
		So, $\ekbar_{i+m}$ occurs in $\ekbar_{i} \ekbar_m$ with coefficient $1$, and there are no terms $\ekbar_k$ with $k \ne i+m$
		in the decomposition of $\ekbar_{i} \ekbar_m$
		from Lemma \ref{pieriformulaktheory}.

From now on, let $l(\nu)=2$. Let us consider the case $i = 1$ first. Then the content of the KOG-tableau is simply $\{1\}$. By Definition \ref{dfn:KOG} (i), each row of $\nu / \lambda$ can have at most 1 box. 
		There are only two partitions $\nu$ that contain $\lambda$ and satisfy these conditions: $\nu = (m,1)$ and $\nu = (m+1,1)$.
		So, $C_{\lambda, 1}^{(m,1)}=C_{\lambda, 1}^{(m+1,1)} = 1$, thus the first equation in the statement of the lemma immediately follows.
		
Let $2\leq i \leq m-1$. If $\abs{\nu/ \lambda}\geq i+2$, then $\abs{\nu}\geq i+m+2$, thus $\ekbar_{\nu}\in K(\bar{X})^{(i+m+2)}$. Therefore, we may assume that $\abs{\nu/ \lambda}=i$ or $i+1$.

Let	$\nu=(j,r)$, where $j>r$ and $j+r=m+i$ or $j+r=m+i+1$. If there is a box in the top row of $\nu / \lambda$ located directly above boxes of the bottom row, then $r=m\geq i+1$, thus $\abs{\nu/ \lambda}\geq i+2$. Hence, it suffices to compute $C_{\lambda, i}^{\nu}$ for the following tableaux of shape $\nu / \lambda$:
		\begin{equation}\label{table4}
		\begin{ytableau}
		\none   & \none  & \none  \\
		b_{1}   & \cdots  & b_{r} 
		\end{ytableau}
		\text{\quad\quad or \quad\quad}
		\begin{ytableau}
		\none  & \none  & \none  & \none[\cdots] & a_{1} & \cdots &a_{k} \\
		b_{1}  & \cdots  & b_{r} & \none & \none & \none & \none ,
		\end{ytableau}
		\end{equation}
		where $k=j-m$, 
		so $k=i-r$ or $k=i+1-r$.
		
	Assume that $k = i - r$. As $m \le j \le i+m-1$, $1 \le r \le i$, we get $0 \le k \le i-1$. We have two subcases: $k=0$ and $1 \le k < i$. If $k = 0$, then by Definition \ref{dfn:KOG} (i) $C_{\lambda, i}^{\nu}=1$. If $1 \le k < i$, then we use induction on $k$: For $k=1$, we get $C_{\lambda, i}^{\nu}=2$ by Example \ref{ex:KOGT} (1). For $k \ge 2$, we have $a_k = i$ by Remark \ref{koggreaterorless}, thus $C_{\lambda, i}^{\nu}=2$ by induction.
		
	Now assume that $k = i+ 1-r$. As the content of $\nu / \lambda$ should be $[1,i]$, it follows from Definition \ref{dfn:KOG} (i) that $r \le i$, and the top row in the tableau (\ref{table4}) is non-empty. So, $1 \le r \le i$, $m+1 \le j \le i+m$, and $1 \le k \le i$. We consider three subcases: $k=1$, $k = i$, and $2 \le k \le i-1$.
		
		If $k=1$, then $r=i$. By Definition \ref{dfn:KOG} (i), there is only one option for the bottom row: $(b_{1},\ldots, b_{i})=(1,\ldots, i)$. By Definition \ref{dfn:KOG} (ii), we have two options for $a_1$: $a_1 = 1$ or $a_1=i$. Hence, we get $C_{\lambda, i}^{\nu}=2$ , i.e., the term $\ekbar_{m+1,\,i}$ occurs in $\ekbar_{i} \ekbar_{m}$ with coefficient $-2$.

		If $k=i$, then $r=1$. By Definition \ref{dfn:KOG} (i), we get $(a_1, \ldots, a_i) = (1, \ldots, i)$. By Definition \ref{dfn:KOG} (ii) applied to $a_2$, we have two options for $b_1$: $b_1=1$ or $b_1=2$. So, $C_{\lambda, i}^{\nu}=2$, and the term 
		$\ekbar_{m+i, \, 1}$ occurs in $\ekbar_{i} \ekbar_{m}$ with coefficient $-2$.
		
		Finally, let $2 \le k \le i-1$. Then, by exactly the same argument as in the last paragraph of the proof of Lemma \ref{ktheorysquaresbarek}, we have $C_{\lambda, i}^{\nu}=3$, thus the term $\ekbar_{m+k,\,i-k+1}$ occurs in $\ekbar_{i} \ekbar_{m}$ with coefficient $-3$ for any $2\leq k\leq i-1$.
		\end{proof}

Lemma \ref{ktheoryproductsbarek} directly implies the following equations in $\widetilde{K}(\bar{X})$.

\begin{cor}\label{ktheoryproductsckcommas}
	Let $m > 1$. Then, we have $f(m)f(1) = f(m+1) + f_{m,\,1} - tf_{m+1,\,1}$ in $\widetilde{K}(\bar{X})$,
	and the following relations hold modulo $I(\bar{X})^{2}:$
\begin{equation*}
		f(m) f(i) \equiv f(m+i)+f_{m,\, i}+2\big(\sum_{k=1}^{i-1}f_{m+k,\, i-k}\big)+t\big(\sum_{k=2}^{i-1}f_{m+k,\, i-k+1}\big)
\end{equation*}
for any $1<i<m$.
\end{cor}

Now using Corollary \ref{ktheoryproductsckcommas}, we get the following intermediate result.

\begin{lemma}\label{ktheorycommasckproducts}
	For any $m>1$, we have the following relation modulo $I (\bar{X})^{2}$ 
	\begin{equation*}
		 f_{m, \, 1}-f(m)f(1) \equiv f(m+1) - t f(1)f(m+1) + t f(m+2).
	\end{equation*}
	For any $1<i<m$, the difference $f_{m, \, i}-f(m)f(i)$ is congruent modulo $I (\bar{X})^{2}$ to
	\begin{equation*}
		(-1)^{i} f(m+i)-2\big(\sum_{k=1}^{i-1}f(m+k)f(i-k)\big)-t\big(\sum_{k=2}^{i-1}f(m+k)f(i-k+1)\big)
	\end{equation*}	
	if $i$ is even, and is congruent modulo $I (\bar{X})^{2}$ to
	\begin{equation*}
		(-1)^{i} f(m+i)-2\big(\sum_{k=1}^{i-1}f(m+k)f(i-k)\big)-t\big(\sum_{k=2}^{i-1}f(m+k)f(i-k+1)\big) - t f(m+i+1)
	\end{equation*}	
	if $i$ is odd.
\end{lemma}
\begin{proof}
	We first observe that $2\equiv -2,\, t\equiv -t \mod I (\bar{X})^{2}$. The formula for $f_{m, \, 1}-f(m)f(1)$ is obtained from 
	the formulas for $f(m)f(1)$ and for $f(m+1) f(1)$ (multiplied by $t$) in Corollary \ref{ktheoryproductsckcommas}.
	
	Let us prove the formula for $f_{m, \, i}-f(m)f(i)$ with $1<i<m$.
	We show by induction on $i$ for all values of $m>i$ together. If $i=2$, then the formula for $f_{m,\,2}-f(m)f(2)$ is obtained from the formulas for $f(m)f(2)$ and for $f(m+1)f(1)$ (multiplied by $2$) in Corollary \ref{ktheoryproductsckcommas}. Now we assume that the formulas for $f_{m', \, i'}-f(m')f(i')$ hold for any $2<i'<i$ and any $m'>i'$. Let us multiply the formulas by $2$ and $t$, respectively. Then, we have the following congruences modulo $I(\bar{X})^2$:
	\begin{equation}\label{ktheoryfmk}
	2f_{m', \, i'}-2f(m')f(i')\equiv 2f(m'+i') \text{ and } tf_{m', \, i'}-tf(m')f(i')\equiv tf(m'+i'),
	\end{equation}
	respectively. Note that by the first formula in Lemma \ref{ktheorycommasckproducts}, the first formula in (\ref{ktheoryfmk}) still holds for $i'=1$ and any $m'>1$.

Taking the sum of the first formulas (\ref{ktheoryfmk}) for $m'=m+k$, $i'=i-k$, and $1 \le k \le i-1$, we get
\begin{equation}
\label{ktheoryfmkintermediateformula1}
2\big(\sum_{k=1}^{i-1}f_{m+k,\,i-k}\big) \equiv 2\big(\sum_{k=1}^{i-1}f(m+k)f(i-k)\big) + 2(i-1) f(m+i) \mod I(\bar{X})^2.
\end{equation}
Similarly, taking the sum of the second formulas (\ref{ktheoryfmk}) for $m'=m+k$, $i'=i-k+1$, and $2 \le k \le i-1$, we get
\begin{equation}
\label{ktheoryfmkintermediateformula2}
t(\sum_{k=2}^{i-1}f_{m+k,\,i-k+1}) \equiv t(\sum_{k=2}^{i-1}f(m+k)f(i-k)) + t(i-2) f(m+i+1) \mod I(\bar{X})^2.
\end{equation}

Let us plug the formulas (\ref{ktheoryfmkintermediateformula1}) and (\ref{ktheoryfmkintermediateformula2}) into the second formula in Corollary \ref{ktheoryproductsckcommas}. Then, the difference $f_{m, \, i}-f(m)f(i)$ is congruent modulo $I(\bar{X})^2$ to
\begin{multline*}
-(2i-1)f(m+i) \\
-t(i-2)f(m+i+1)-2\big(\sum_{k=1}^{i-1}f(m+k)f(i-k)\big)-t\big(\sum_{k=2}^{i-1}f(m+k)f(i-k+1)\big).
\end{multline*}	
Since the following congruences hold modulo $I(\bar{X})^2$
	\begin{equation}\label{ktheoryparity1}
-(2i-1)
\equiv (-1)^{i}\quad \text{ and }\quad	
-
t(i-2) \equiv
	\begin{cases}
	0 & \text{ if } i \text{ is even},\\
	t & \text{ if } i\text{ is odd},
	\end{cases}
	\end{equation}
the formula follows.\end{proof}

Combining Corollary \ref{ktheorysquaresckcommas} and Lemma \ref{ktheorycommasckproducts}, we obtain the following main result of this section. 

\begin{prop}\label{ktheorysquaresckproducts}
For any $1<i<n$, the following relations hold modulo $I (\bar{X})^{2}:$
	\begin{equation*}
		f(i)^{2} \equiv
		(-1)^{i-1} f(2i)+2\big(\sum_{k=1}^{i-1} f(i+k)f(i-k)\big)-  t \big(\sum_{k=1}^{i-1}f(i+k)f(i-k+1)\big)  +t f(2i+1)
	\end{equation*}	
	for even $i$, and
	\begin{equation*}
		f(i)^{2} \equiv
		(-1)^{i-1} f(2i)+2\big(\sum_{k=1}^{i-1} f(i+k)f(i-k)\big)-  t \big(\sum_{k=1}^{i-1}f(i+k)f(i-k+1)\big)
			\end{equation*}	
	for odd $i$.
\end{prop}
\begin{proof}
Let us rewrite the formulas from Lemma \ref{ktheorycommasckproducts} as follows:
\begin{equation*}
f_{m', \, i'} \equiv f(m')f(i') + f(m'+i') - t f(i')f(m'+i') + t f(m'+i'+1)
\end{equation*}
for $m' > 1, i'=1$, 
\begin{multline*}
		f_{m', \, i'} \equiv f(m')f(i') + (-1)^{i'} f(m'+i')\\
		-2\big(\sum_{k=1}^{i'-1}f(m'+k)f(i'-k)\big)-t\big(\sum_{k=2}^{i'-1}f(m'+k)f(i'-k+1)\big)
	\end{multline*}
for $m' > i' > 1$, $i'$ even, and
	\begin{multline*}
		f_{m', \, i'} \equiv f(m')f(i') + (-1)^{i'} f(m'+i')\\
		-2\big(\sum_{k=1}^{i'-1}f(m'+k)f(i'-k)\big)-t\big(\sum_{k=2}^{i'-1}f(m'+k)f(i'-k+1)\big) - t f(m'+i'+1)
	\end{multline*}	
for $m' > i' > 1$, $i'$ odd, where all congruences are modulo $I(\bar X)^{2}$.

Multiplying each of these formulas by 2, for any $i'\ge 1$ and any $m' > i' \ge 1$ we have 
\begin{equation}
\label{ktheorycommasckproductstimes2}
2 f_{m', \, i'} \equiv 2 f(m')f(i') + 2 f(m'+i') \mod I(\bar{X})^2.
\end{equation}
Similarly, multiplying by $t$, for any $m' > i' \ge 1$ we get
\begin{equation}
\label{ktheorycommasckproductstimest}
t f_{m', \, i'} \equiv t f(m')f(i') + t f(m'+i') \mod I(\bar{X})^2.
\end{equation}

For any $1<i<n$, let us take the sum of (\ref{ktheorycommasckproductstimes2}) for $m'=i+k$ and $i'=i-k$ over $1\leq k\leq i-1$. Then, we get
		\begin{equation*}
		\label{ktheoryfksquareintermediateformula1}
		2\big(\sum_{k=1}^{i-1}f_{i+k,\,i-k}\big) \equiv 2\big(\sum_{k=1}^{i-1}f(i+k)f(i-k)\big) + 2(i-1) f(2i) \mod I(\bar{X})^2.
		\end{equation*}
Similarly, we take the sum of (\ref{ktheorycommasckproductstimest}) for $m'=i+k$ and $i'=i-k+1$ over $1\leq k\leq i-1$. Then, we have
		\begin{equation*}
		t\big(\sum_{k=1}^{i-1}f_{i+k,\,i-k+1}\big) \equiv t\big(\sum_{k=1}^{i-1}f(i+k)f(i-k+1)\big) + t(i-1) f(2i+1) \mod I(\bar{X})^2.
		\end{equation*}

Now let us plug these formulas into the statement of Corollary \ref{ktheorysquaresckcommas} for $1 < i < n$, thus we have the following congruence modulo $I(\bar{X})^{2}$ 
\begin{equation*}
		f(i)^{2} \equiv
		2\big(\sum_{k=1}^{i-1} f(i+k)f(i-k)\big)-  t \big(\sum_{k=1}^{i-1}f(i+k)f(i-k+1)\big) + (2i-1)f(2i) -  t(i-1) f(2i+1).
\end{equation*}
Since $-t(i-1)\equiv t \mod I(\bar{X})^{2}$ if $n$ is even, and $-t(i-1)\equiv  0\mod I(\bar{X})^{2}$ otherwise, the statement follows from the first congruence equation in (\ref{ktheoryparity1}).
\end{proof}

\end{document}